\documentclass{amsart}
\usepackage{amscd,amsmath,amssymb,amsthm,bm,cases,color,comment,mathrsfs}
\usepackage[pdftex]{graphicx}
\usepackage[all]{xy}

\makeatletter
\@namedef{subjclassname@2020}{%
\textup{2020} Mathematics Subject Classification}
\makeatother

\theoremstyle{definition}
\newtheorem{definition}{Definition}[section]
\newtheorem{remark}[definition]{Remark}

\theoremstyle{plain}
\newtheorem{theorem}[definition]{Theorem}
\newtheorem{proposition}[definition]{Proposition}
\newtheorem{lemma}[definition]{Lemma}
\newtheorem{corollary}[definition]{Corollary}
\newtheorem{claim}[definition]{Claim}

\numberwithin{equation}{section}

\title[Virtual embeddings of braid groups]{On virtual embeddings of braid groups into mapping class groups of surfaces}

\author[T.~Katayama]{Takuya Katayama}
\address{
(Takuya Katayama)
Department of Mathematics, 
Faculty of Science, 
Gakushuin University, 
1-5-1 Mejiro, Toshima-ku, Tokyo 171-8588, Japan
}
\email{katayama@math.gakushuin.ac.jp}
\author[E.~Kuno]{Erika Kuno}
\address{
(Erika Kuno)
Department of Mathematics,
Graduate School of Science, 
Osaka University,
1-1 Machikaneyama-cho Toyonaka, Osaka 560-0043, Japan
}
\email{e-kuno@math.sci.osaka-u.ac.jp}
\date{\today}

\date{\today}

\begin{document}

\begin{abstract}
In this article, we give a necessary and sufficient condition for embedding a finite index subgroup of Artin's braid group into the mapping class group of a connected orientable surface. 
\end{abstract}

\maketitle


\section{Introduction \label{Introduction_section}}

Let $S_{g, p}^{b}$ be a connected orientable surface of genus $g$ with $p$ punctures and $b$ boundary components. 
In the case where $b=0$ or $p=0$, we drop the suffix that denotes $0$, excepting $g$, from
$S_{g, p}^{b}$. 
For example, $S_{g, p}^{0}$ will simply be denoted as $S_{g, p}$. 
The {\it mapping class group} $\mathrm{Mod}(S_{g, p}^{b})$ of $S_{g, p}^{b}$ is the group of orientation-preserving homeomorphisms of $S_{g, p}^{b}$, fixing the boundary pointwise, up to isotopy relative to the boundary. 
We write $B_{n}$ for the {\it braid group} on $n$ strands, which is identified with $\mathrm{Mod}(S_{0, n}^{1})$. 
We define the Euler charcteristic of $S_{g, p}^{b}$ as 
$$ \chi_{g, p}^{b} := 2- 2g -p - b  .$$ 
We say that a group $H$ is {\it virtually embedded} in a group $G$ if $H$ has a finite index subgroup that injects into $G$. 
An injective homomorphism from a finite index subgroup of $H$ to $G$ is called a {\it virtual embedding} of $H$ (into $G$). 
In this paper we give a necessary and sufficient condition for the existence of virtual embeddings of braid groups into the mapping class groups of orientable surfaces. 

\begin{theorem}
Suppose $b=0$. 
Then the braid group $B_{n}$ is virtually embedded in $\mathrm{Mod}(S_{g, p})$ if and only if 
\begin{eqnarray*}
n \leq \left\{ \begin{array}{ll}
1 & ((g, p) \in \{ (0, 0), (0, 1), (0, 2), (0, 3) \}) \\
2 & ((g, p) \in \{ (0, 4), (1, 0), (1, 1) \}) \\
- \chi_{g, p} & (g=0, \ p \geq 5) \\
2 - \chi_{g, p} & (g \geq 2, \ p=0)  \\
1 - \chi_{g, p} & (\mbox{otherwise}). \\
\end{array} \right.
\end{eqnarray*}
\label{main_1}
\end{theorem}

\begin{theorem}
Suppose $b \geq 1$. 
Then $B_{n}$ is virtually embedded in $\mathrm{Mod}(S_{g, p}^{b})$ if and only if 
\begin{eqnarray*}
n \leq \left\{ \begin{array}{ll}
2 - \chi_{g, p}^{b} & (g \geq 1, \ p+b \leq 2) \\
1 - \chi_{g, p}^{b} & (\mbox{otherwise}). \\
\end{array} \right.
\end{eqnarray*}
\label{main_2}
\end{theorem}

If a group $L$ is virtually embedded in a group $H$ and $H$ is virtually embedded in a group $G$, then $L$ is virtually embedded in $G$. 
Note as well that virtual embeddings are closely related to homomorphisms with finite kernel. 
In fact, the mapping class group $\mathrm{Mod}(S)$ is virtually embedded in a group $G$ if and only if there is a homomorphism with finite kernel from a finite index subgroup of $\mathrm{Mod}(S)$ to $G$, because $\mathrm{Mod}(S)$ is virtually torsion-free by Serre's theorem. 
We are often interested in the {\it pure mapping class group}, denoted by $\mathrm{PMod}(S_{g, p}^{b})$, which is a subgroup of $\mathrm{Mod}(S_{g, p}^{b})$ and consists of the mapping classes fixing the punctures. 
We denote the {\it pure braid group} $\mathrm{PMod}(S_{0, n}^{1})$ by $PB_n$; this is an index $n!$ subgroup of $B_n$. 

The study of injective homomorphisms between the mapping class groups of surfaces has been extensively developed by various researchers. 
Aramayona--Leininger--Souto \cite{Aramayona--Leininger--Souto06}, Birman--Hilden \cite{Birman--Hilden} and Paris--Rolfsen \cite{Paris--Rolfsen00} gave injective homomorphisms induced by inclusion maps and (possibly branched) coverings of surfaces in an ingenious manner. 
Aramayona--Souto \cite{Aramayona--Souto} proved that $\mathrm{PMod}(S_{g, p}^{b})$ cannot be embedded in $\mathrm{PMod}(S_{g', p'}^{b'})$ if $g \geq 6$ and either $g' \leq 2g-2$ or $g'=2g-1$ and $p'+b' \geq 1$. 
Ivanov--McCarthy \cite{Ivanov--McCarthy} proved that every injective homomorphism between the mapping class groups of surfaces with almost the same topological complexity is necessarily an isomorphism. 
Shackleton \cite{Shackleton} proved that Ivanov--McCarthy's result is also true for virtual embeddings. 
The topological complexity of surfaces \cite{Birman--Lubotzky--McCarthy} and virtual co-homological dimension \cite{Harer} are known to be obstructions to the existence of virtual embeddings. 
It follows from these results that, apart from a few exceptions, two mapping class groups are abstractly commensurable if and only if the defining surfaces are homeomorphic. 
Moreover, the rigidity result of the mapping class groups due to Behrstock--Kleiner--Minsky--Mosher \cite{Behrstock--Kleiner--Minsky--Mosher} and Hamenst\"adt \cite{Hamenstaedt} implies that two mapping class groups are quasi-isometric if and only if they are abstractly commensurable. 
Therefore, finite index subgroups of the mapping class groups of surfaces reflect the homeomorphism types of the surfaces. 
As for braid groups, Castel \cite{Castel16} characterized homomorphisms from braid groups to the mapping class groups of surfaces without punctures and proved that every injective homomorphism is a transvection of a geometric monodromy. 

\begin{remark}
From Castel's work \cite{Castel16}, it follows that $B_{2g+3}$ is not embedded in $\mathrm{Mod}(S_{g}^{b})$ for any $g \geq 0$ and $b \geq 0$. 
On the other hand, Theorem \ref{main_2} says that $B_{2g+3}$ has a finite index subgroup embedded in $\mathrm{Mod}(S_{g}^{b})$ if $b \geq 4$. 
Therefore, if $b \geq 4$, then there is a virtual embedding of $B_{2g+3}$ into $\mathrm{Mod}(S_{g}^{b})$ that does not extend to an embedding of the entire group $B_{2g+3}$. 
\end{remark}

Recently, Chen--Mukherjea in \cite{Chen--Mukherjea} proved that any non-trivial homomorphism from a braid group to the pure mapping class groups of a surface with punctures is a transvection of a geometric monodromy if the genus of the surface lies in a certain range. 

This paper is organized as follows. 
In Section \ref{const_embeddings}, we review the embeddings between the mapping class groups of surfaces that have been obtained from the Birman--Hilden theory and Paris--Rolfsen's result. 
We then introduce the concept of ``psuedo-annular extensions" of surfaces and prove that ``pseudo-annular extensions" of spheres induce virtual embeddings of the mapping class groups of the spheres. 
These virtual embeddings of the mapping class groups of surfaces enable us to prove the ``if part" of Theorems \ref{main_1} and \ref{main_2} in the last section of this paper. 
On the other hand, the ``only if part" of Theorems \ref{main_1} and \ref{main_2} will be proved by comparing right-angled Artin subgroups in the mapping class groups of surfaces. 
In Sections \ref{cyc_chain_surf} and \ref{raags_in_mcgs}, we determine whether a right-angled Artin group $A$ of specific type is embedded in the mapping class group of a surface $S$. 
More precisely, in Section \ref{cyc_chain_surf}, we decide whether the defining graph of $A$ and its companions are full subgraphs of the $1$-skeleton of the Harvey's curve complex of $S$. 
In Section \ref{raags_in_mcgs}, we prove that $A$ is embedded in the mapping class group of $S$ if and only if the defining graph of $A$ is a full subgraph of the $1$-skeleton of Harvey's the curve complex of $S$ (Remark \ref{iff_raag}). 
The last section of this paper is devoted to proving Theorems \ref{main_1} and \ref{main_2}.

\subsection*{Acknowledgements}
The authors are grateful to Donggyun Seo and Dan Margalit for helpful comments. 
The first author was supported by JSPS KAKENHI through a grant, number 20J01431. 
The second author was supported by JSPS KAKENHI Grant-in-Aid for Young Scientists, number 21K13791.

\section{Constructions of virtual embeddings \label{const_embeddings}}

In this section, we explain how to obtain virtual embeddings of braid groups into the mapping class groups of surfaces. 
Our constructions of the virtual embeddings are based on extensions and coverings of surfaces. 

First, let us review homomorphisms induced by embeddings of surfaces. 
Let $i \colon S \rightarrow F$ be an inclusion map between two connected orientable surfaces of finite type. 
We call $i$ an {\it extension} of $S$ for convenience sake. 
The extension $i$ is said to be {\it admissible} if $i(S)$ is a closed subset of $F$ and no component of $i(\partial S)$ is parallel to a component of $\partial F$. 
Additionally, we say that an admissible extension $i$ is {\it annular} if $F \setminus \mathrm{Int}(i(S))$ is a disjoint union of annuli.   

\begin{proposition} 
Suppose that $i \colon S \rightarrow F$ is an admissible extension. 
\begin{enumerate}
 \item[(1)] If the extension $i$ is annular, then the kernel of the induced homomorphism $\mathrm{Mod}(S) \rightarrow \mathrm{Mod}(F)$ is a free abelian group generated by 
$$\mathcal{A} = \{ [T_{c_1}] [T_{c_2}]^{-1} \mid c_1, c_2 \mbox{ are the boundary components of an outer annulus} \}.$$
 \item[(2)] If $F \setminus \mathrm{Int}(i(S))$ is a disjoint union of annuli and once-punctured disks, then the kernel of the induced homomorphism $\mathrm{Mod}(S) \rightarrow \mathrm{Mod}(F)$ is a free abelian group generated by $\mathcal{A}$ and 
$$\{ [T_{c}] \mid c \mbox{ is the boundary component of an outer once-punctured disk} \}.$$
\end{enumerate}
Here, $T_{c}$ is a Dehn twist about a curve $c$. 
\label{annular_kernel}
\end{proposition}

When each component of the exterior of an admissible extension is sufficiently complicated, the extension induces an embedding of the mapping class group. 

\begin{proposition} 
Suppose that $F \setminus i(S)$ is a disjoint union of surfaces of negative Euler characteristics. 
Then $\mathrm{Mod}(S) \hookrightarrow \mathrm{Mod}(F)$. 
\label{hyp_emb}
\end{proposition}

Propositions \ref{annular_kernel} and \ref{hyp_emb} were proved by Paris--Rolfsen in \cite{Paris--Rolfsen00}. 
We say that an admissible extension $i$ is {\it hyperbolic} if $i$ satisfies the assumption in Proposition \ref{hyp_emb}. 
The embedding given in the following corollary is induced by a hyperbolic extension $S_{0, n}^{1} \rightarrow S_{0, m}^{1}$. 

\begin{corollary}
If $n \leq m$, then $B_n \hookrightarrow B_m$. 
\label{braid_braid}
\end{corollary}

Now, let us recall some injective homomorphisms induced by coverings of surfaces. 
Consider the double-branched covering $S_{g-1, 0}^{2} \rightarrow S_{0, 2g}^{1}$ induced by the hyperelliptic involution of $S_{g-1, 0}^{2}$. 
According to Birman--Hilden theory, it follows that $B_{2g}$ is embedded in $\mathrm{Mod}(S_{g-1}^{2})$ as the symmetric mapping class subgroup with respect to the hyperelliptic involution. 
By gluing $S_{g-1}^{2}$ and $S_{0, p}^{b+2}$ along their boundaries, we obtain an admissible extension $S_{g-1}^{2} \rightarrow S_{g, p}^{b}$. 
This extension is either annular or hyperbolic, thereby inducing a homomorphism $\mathrm{Mod}(S_{g-1}^{2}) \rightarrow \mathrm{Mod}(S_{g, p}^{b})$, whose restriction to the subgroup $B_{2g}$ is injective. 
Therefore, $B_{2g}$ is embedded in $\mathrm{Mod}(S_{g, p}^{b})$. 
For more information about Birman--Hilden theory, see \cite{Birman--Hilden} and \cite{Margalit--Winarski}. 

\begin{proposition}
Suppose $g \geq 1$, $p \geq 0$ and $b\geq 0$. 
Then $B_{2g} \hookrightarrow \mathrm{Mod}(S_{g, p}^{b})$. 
\end{proposition}

Similarly, from the double-branched covering $S_{g}^{1} \rightarrow S_{0, 2g+1}^{1}$, we obtain $B_{2g+1} \hookrightarrow \mathrm{Mod}(S_{g}^{1})$. 
The target of this embedding can be extended as follows. 

\begin{proposition}
Suppose $g \geq 1$ and $b+p \geq 2$. 
Then $B_{2g+1} \hookrightarrow \mathrm{Mod}(S_{g, p}^{b})$. 
\end{proposition}
\begin{proof}
Since $b+p \geq 2$, an extension $S_{g}^{1} \rightarrow S_{g, p}^{b}$ is hyperbolic. 
Hence, we have $B_{2g+1} \hookrightarrow \mathrm{Mod}(S_{g}^{1}) \hookrightarrow \mathrm{Mod}(S_{g, p}^{b})$. 
\end{proof}

In the case where $g \geq 1$, $b=1$ and $p=1$, we can find a larger braid group in $\mathrm{Mod}(S_{g, 1}^{1})$. 

\begin{proposition} 
Suppose $g \geq 1$. 
Then $B_{2g+2} \hookrightarrow \mathrm{Mod}(S_{g, 1}^{1})$. 
 \label{for_ex_case}
\end{proposition}
\begin{proof}
We use the embedding $B_{2g+2} \hookrightarrow \mathrm{Mod}(S_{g}^{2})$ induced by the covering described above and the restriction of the capping homomorphism $\varphi \colon \mathrm{Mod}(S_{g}^{2}) \rightarrow \mathrm{Mod}(S_{g, 1}^{1})$ to $B_{2g+2}$. 
Since no non-trivial element of $\mathrm{Ker}(\varphi)$ preserves the fiber of a point in $S_{0, 2g+1}^{1}$, the restriction is injective. 
\end{proof}

The following concept is fairly important for finding the largest virtual embeddings of braid groups into mapping class groups. 

\begin{definition}
Let $S$ be a surface with punctures and $F$ a surface. 
By $\overline{S}$, we denote the compactification of $S$ that has circles at infinity (if $S \cong S_{g, p}^{b}$, then $\overline{S} \cong S_{g}^{b+p}$). 
An extension $i \colon S \rightarrow F$ is said to be {\it pseudo-annular} if $i$ is obtained from a disjoint union of $\overline{S}$ and copies of an annulus and a once-punctured disk by gluing each annulus $A$ (resp.\ once-punctured disk $D$) to $\overline{S}$ along their boundaries so that one boundary component of $A$ (resp.\ the boundary of $D$) is identified with a component of $\overline{S} - S$ and that another boundary component of $A$ is either identified with a component of $\partial \overline{S}$ or is unattached. 
Note that no pseudo-annular extension is admissible, because its image is not a closed subset. 
However, $\overline{S} \subset F$ is admissible. 
\end{definition}

The pure mapping class group of a sphere is naturally embedded in the mapping class group of its compactification as follows. 

\begin{lemma}
Suppose $p+b \geq 3$. 
Then $\mathrm{PMod}(S_{0, p}^{b}) \times \mathbb{Z}^{p} \cong \mathrm{PMod}(S_{0}^{p+b})$. 
Moreover, each standard generator of the direct factor $\mathbb{Z}^{p}$ corresponds to a Dehn twist about a closed curve parallel to a boundary component of $S_{0}^{p+b}$. 
 \label{p_to_b_sphere}
\end{lemma}
\begin{proof}
See \cite[Theorem 10]{Clay--Leininger--Margalit}. 
\end{proof}

A pseudo-annular extension of a sphere induces an embedding of the pure mapping class group of the sphere. 

\begin{proposition} 
Let $S \rightarrow F$ be a pseudo-annular extension. 
If $S$ is a sphere, then $\mathrm{PMod}(S)$ is embedded in $\mathrm{PMod}(F)$. 
 \label{pse_ann_emb}
\end{proposition}
\begin{proof}
The admissible extension $\overline{S} \subset F$ induces an embedding $\phi \colon \mathrm{Homeo}(\overline{S}) \hookrightarrow \mathrm{Homeo}(F)$. 
Then $\phi$ gives a well-defined homomorphism $\varphi \colon \mathrm{Mod}(\overline{S}) \rightarrow \mathrm{Mod}(F)$ whose kernel is generated by Dehn twists corresponding to the attached once-punctured disks and products of Dehn twists corresponding to the attached annuli. 
Since $S$ is a sphere, Lemma \ref{p_to_b_sphere} implies that $\mathrm{PMod}(S) \hookrightarrow \mathrm{Mod}(\overline{S})$. 
Since $S \rightarrow F$ is pseudo-annular, Proposition \ref{annular_kernel} together with Lemma \ref{p_to_b_sphere} implies that the intersection of $\mathrm{Ker}(\varphi)$ and $\mathrm{PMod}(S)$ is trivial. 
Therefore, the restriction of $\varphi$ to $\mathrm{PMod}(S)$ is injective. 
\end{proof}

\section{Cyclic chains on Surfaces \label{cyc_chain_surf}}
In this section we investigate combinatorial obstructions, contained in the $1$-skeleta of Harvey's curve complexes, to embedding certain right-angled Artin groups into the mapping class groups of surfaces. 
More precisely, we give a necessary and sufficient condition for embedding $C_m^c$, $C_m^c * K_1$ and some graphs similar to these into $\mathcal{C}(S)$ as ``full subgraphs"; see Lemmas \ref{cyclic_full_subgraph}--\ref{2g_p}. 
Here, $C_m$ is a cyclic graph of length $m$ (the geometric realization is homeomorphic to the unit circle), $\Gamma^c$ is the complement graph of a given graph $\Gamma$, $K_m$ is a complete graph on $m$ vertices, and $\mathcal{C}(S)$ is the curve graph of a surface $S$ that is the $1$-skeleton of Harvey's curve complex. 

An arc or closed curve on a surface is called {\it simple} if it has no self-intersection. 
An arc $\delta$ on a surface $S$ is {\it properly embedded} if $\partial \delta \subseteq \partial S$ and $\delta$ is a transversal to $\partial S$. 
A properly embedded arc $\delta$ on $S$ is called {\it essential} if it is not isotopic to $\partial S$.
A simple closed curve $\alpha$ on $S$ is called {\it essential} if it does not bound a disk and it is not isotopic to a boundary component of $S$. 
For short, we use will the term arc (resp.\ closed curve) to mean a simple, properly embedded and essential arc (resp.\ closed curve) unless otherwise noted. 

\begin{definition}
A sequence $( \alpha_1, \ldots, \alpha_m )$ of closed curves on a surface $S$ is called a {\it semi-cyclic chain} if any two consecutive curves $\alpha_{i}$ and $\alpha_{i+1}$ are in minimal position and not disjoint ($i= 1, \ldots, m$ mod $m$). 
If $( \alpha_1, \ldots, \alpha_m )$ is a semi-cyclic chain, we call $m$ its {\it length}. 
We say that a semi-cyclic chain $(\alpha_1, \ldots, \alpha_m )$ on a surface has a {\it chord} if there are non-consecutive curves $\alpha_{i}$ and $\alpha_{j}$ that are in minimal position and not disjoint ($j \neq i, j \neq i \pm 1$ mod $m$). 
A semi-cyclic chain without a chord is called a {\it cyclic chain}. 
Note that if $( \alpha_1, \ldots, \alpha_m )$ is a cyclic chain, then every length $(m-1)$ subsequence of the form $( \alpha_{i+1}, \ldots, \alpha_{m}, \alpha_1, \ldots \alpha_{i-1} )$ is a linear chain ($i=1, \ldots, m$ mod $m$) as defined in \cite[Section 2]{Katayama--Kuno18}.  
\end{definition}

Let $\Lambda$ and $\Gamma$ be graphs. 
We denote the vertex set of the graph $\Lambda$ (resp.\ the edge set of $\Lambda$) by $V(\Lambda)$ (resp.\ $E(\Lambda)$). 
We use the term map, from a graph $\Lambda$ to a graph $\Gamma$, to mean a map from $V(\Lambda)$ to $V(\Gamma)$.  
A map $\phi \colon \Lambda \rightarrow \Gamma$ is called a ({\it graph}) {\it homomorphism} if it is a product map from $\Lambda= (V(\Lambda), E(\Lambda))$ to $\Gamma = (V(\Gamma), E(\Gamma))$ such that $\phi([v_1, v_2]) = [\phi(v_1), \phi(v_2)]$ for all pair of adjacent vertices $v_1 , v_2$ of $\Lambda$. 
Here, we use the notation $[\bullet, \bullet]$ to mean the edge joining two vertices in the bracket. 
A {\it (graph)} {\it embedding} is an injective graph homomorphism. 
The image of a graph embedding is called a {\it subgraph}. 
A map $\phi \colon \Lambda \rightarrow \Gamma$ is said to be {\it full} if it is a product map from $\Lambda$ to $\Gamma$ such that $\phi([v_1, v_2]) = [u_1, u_2]$ for all pairs of adjacent vertices $u_1 , u_2$ of $\Gamma$ and for all $v_1 \in \phi^{-1}(u_1)$ and $v_2 \in \phi^{-1}(u_2)$. 
We write $\Lambda \leq \Gamma$ if there is a full embedding of $\Lambda$ into $\Gamma$. 
In this case, both $\Lambda$ and the image of the full embedding are called a {\it full subgraph} of $\Gamma$. 
Note that any full subgraph of a graph is uniquely determined by its vertex set. 
For a subset $G \subset \Gamma$, the {\it full subgraph induced by $G$} is the full subgraph of $\Gamma$, whose vertex set is identical to that of $G$. 

By definition, we have the following. 

\begin{proposition}
There is a cyclic chain of length $m$ on a surface $S$ of negative Euler characteristic if and only if $C_m^c \leq \mathcal{C}(S)$. 
\end{proposition}

For the purpose of proving the non-existence of cyclic chains on surfaces, we introduce the following concept. 

\begin{definition}
Let $\delta_{a}$ and $\delta_{b}$ be arcs on a surface $S$, and $( \alpha_1, \ldots, \alpha_m )$ a linear chain on $S$. 
A triple $(\delta_{a}, (\alpha_1, \ldots, \alpha_m ), \delta_{b})$ is called {\it chained} if the triple satisfies the following conditions: 
\begin{enumerate}
 \item[(1)] $\delta_{a}$ intersects $\alpha_1$ non-trivially in minimal position but is disjoint from $\alpha_2, \ldots, \alpha_m$ and $\delta_{b}$. 
 \item[(2)] $\delta_{b}$ intersects $\alpha_m$ non-trivially in minimal position but is disjoint from $\alpha_1, \ldots, \alpha_{m-1}$ and $\delta_{a}$. 
\end{enumerate}
We call $m$ the {\it length} of the chained triple. 
We say that an arc on $S$ is {\it recursive} if the arc joins a single boundary component of $S$ to itself. 
Note that a separating arc on a surface (i.e., the surface minus the arc is not connected) is recursive. 
A chained triple $(\delta_{a}, (\alpha_1, \ldots, \alpha_m ), \delta_{b})$ on a surface $S$ satisfies condition ($*$) if the following holds: for any boundary component $C$ of $S$, we have either $\delta_a \cap C = \emptyset$ or $\delta_b \cap \emptyset$. 
\end{definition}

We will frequently use the following two lemmas. 

\begin{lemma}
If $\delta$ is an arc on $S_{g}^{p}$, then 
 \begin{eqnarray*}
S_{g}^{p} \setminus \mathrm{Int}N(\delta) \cong \left\{ \begin{array}{ll}
S_{g-1}^{p+1} \ & (\delta: \mbox{ non-separating}), \\
S_{g_1}^{p_1} \sqcup S_{g_2}^{p_2}  \ & (\delta: \mbox{ separating}), \\ 
\end{array} \right.
\end{eqnarray*}
where $g_{1}$, $p_{1}$, $g_{2}$ and $p_{2}$ are natural numbers satisfying 
 \begin{itemize}
 \item[$\bullet$] $g_1 + g_2 = g$, $g_1 \geq 0$, $g_2 \geq 0$,
 \item[$\bullet$] $p_1 + p_2 = p + 1$, $p_1 \geq 1$, $p_2 \geq 1$,
  \item[$\bullet$] if $g_i=0$, then $p_i \geq 2$ {\rm (}$i=1,2${\rm )}.
 \end{itemize}
\label{cutting_surfaces_arc}
\end{lemma}

\begin{lemma}
Let $\mathcal{T} = (\delta_{1}, ( \alpha_1, \ldots, \alpha_m ), \delta_{2})$ be a chained triple on $S_{g}^{p}$. 
Suppose that $S'$ is a connected component of $S_{g}^{p} \setminus \mathrm{Int}N(\delta_{2})$, which contains a (possibly not chained) triple $\mathcal{T}' = (\delta_{1}, ( \alpha_1, \ldots, \alpha_{m-1} ), \delta'_2)$, where $\delta'_{2}$ is a connected component of $\alpha_m \cap S'$ intersecting $\alpha_{m-1}$. 
Then, we have the following. 
\begin{enumerate}
 \item[(1)] If there is a closed curve $\alpha_i$ which is parallel to a boundary component of $S'$, then $i=1$ and $m=2$. 
 \item[(2)] If $m \geq 3$, then $\mathcal{T}' = (\delta_{1}, (\alpha_1, \ldots, \alpha_{m-1} ), \delta'_{2})$ is a chained triple on  $S'$. 
 \item[(3)] If $\delta_1$ is recursive and $\delta_2$ is either separating (i.e., $S_{g}^{p} \setminus \mathrm{Int}N(\delta_{2})$ is not connected) or non-separating and not recursive on $S_{g}^{p}$, then $\delta_1$ is recursive on $S'$. 
 \item[(4)] If $\delta_1$ is either separating or non-separating and not recursive, then $\delta'_2$ is recursive on $S'$. 
 \item[(5)] If $\mathcal{T}$ is chained triple satisfying condition ($*$), then for every boundary component $C$ of $S'$, we have either $\delta_1 \cap C = \emptyset$ or $\delta'_2 \cap C = \emptyset$. 
\end{enumerate}
\label{key_property_of_chained_triples}
\end{lemma}
\begin{proof}
(1) Suppose that a curve $\alpha_i$ is isotopic to a boundary component of $S'$. 
Then, since $\alpha_i$ is essential in $S_{g}^{p}$, $\alpha_i$ is isotopic to $\partial S_{g}^{p} \cup \delta_{2}$. 
Then, depending on whether $\partial \delta_{2}$ lies in a single boundary component of $S_{g}^{p}$ or $\partial \delta_{2}$ is not contained in a single boundary component, we have the following. 
\begin{enumerate}
 \item[(a)] If $\delta_{2}$ joins a boundary component $C$ to itself, then there is an annulus $A$ in $S_{g}^{p}$ such that $\partial A \subset \alpha_i \sqcup (\delta_{2} \cup  \partial S_{g}^{p})$ (see Figure \ref{annulus_disk} (a)). 
 \item[(b)] If $\delta_{2}$ joins a boundary component $C_1$ to another boundary component $C_2$, then there is a two-holed disk $D$ in $S_{g}^{p}$ such that $ \partial D = \alpha_i \sqcup C_1 \sqcup C_2$ (see Figure \ref{annulus_disk} (b)). 
\end{enumerate}
\begin{figure}
\centering
\includegraphics[clip, scale=0.35]{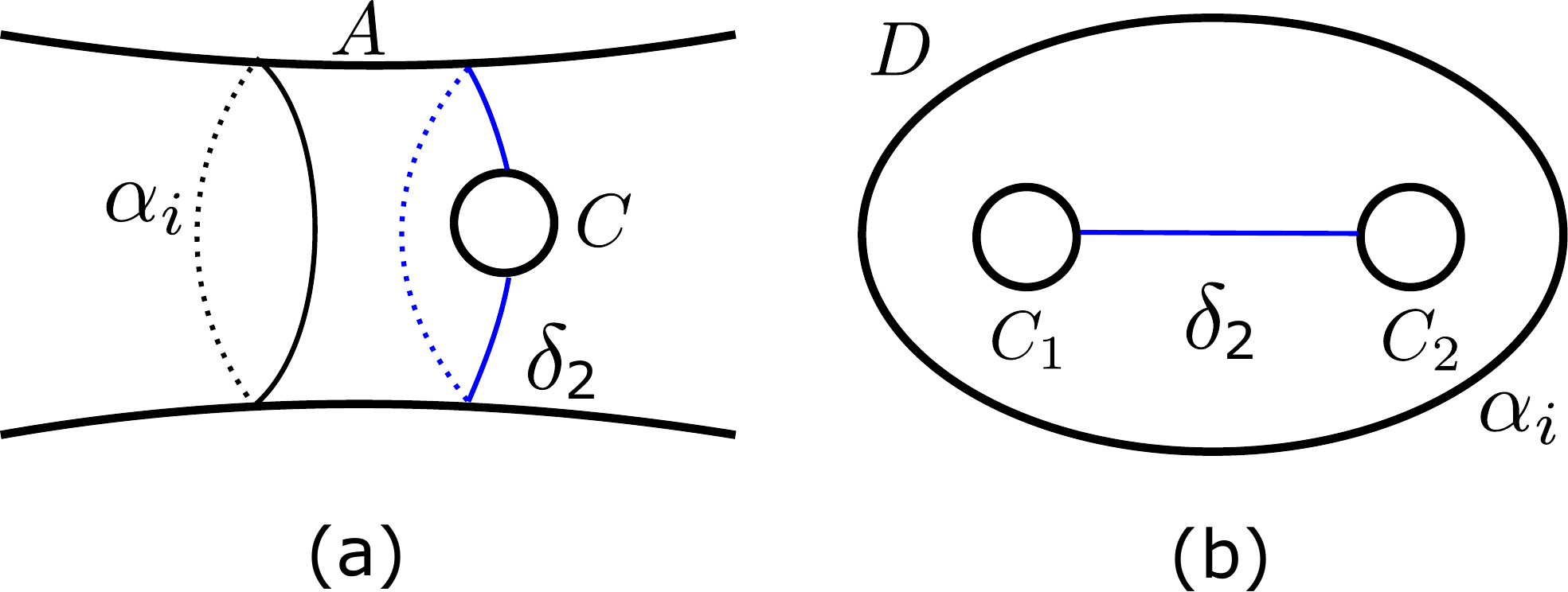}
\caption{(a) the arc $\delta_2$ joins $C$ to itself, (b) $\delta_2$ joins two distinct boundary components. \label{annulus_disk}}
\end{figure}
In each case, any closed curve in $S_{g}^{p}$ that intersects $\alpha_i$ non-trivially and minimally must intersect with $\delta_{2}$. 
Hence, we have $i=1$ and $m=2$. 

(2) Suppose that $m \geq 3$. 
Then, any closed curve in $( \alpha_1, \ldots, \alpha_{m-1} )$ is not boundary parallel (by (1)) and bounds no disk, and hence, it is essential in $S'$. 
In addition, consecutive closed curves $\alpha_i$ and $\alpha_{i+1}$ intersect minimally and non-trivially, and non-consecutive closed curves are disjoint. 
Therefore, $( \alpha_1, \ldots, \alpha_{m-1} )$ is a linear chain on $S'$. 
From our assumption that $(\delta_{a}, ( \alpha_1, \ldots, \alpha_m ), \delta_{2})$ is a chained triple and $\delta_b' \subset \alpha_m$, the arc $\delta_{1}$ intersects with $\alpha_1$ minimally and non-trivially but does not intersect $\alpha_2, \ldots, \alpha_{m-1}$ or $\delta'_{2}$. 
Similarly, the arc $\delta'_{2}$ intersects $\alpha_{m-1}$ minimally and non-trivially but does not intersect $\alpha_1, \ldots, \alpha_{m-2}$ or $\delta_{1}$. 
Thus, $(\delta_{1}, ( \alpha_1, \ldots, \alpha_{m-1} ), \delta'_{2})$ is a chained triple on $S'$. 

(3)  Suppose that $\delta_2$ is either separating or non-separating and not recursive on $S_{g}^{p}$. 
Then, in either case, for each component $S''$ of $S_{g}^{p} \setminus \mathrm{Int}N(\delta_2)$, the number of boundary components of $S''$ that are derived from $\delta_2$ (and a boundary component of $S_{g}^{p}$) is exactly one. 
Hence, this fact together with the assumption that $\delta_1$ is recursive on $S_{g}^{p}$ implies that $\delta_1$ is recursive on $S'$. 

(4) Suppose that $\delta_2$ is either separating or non-separating and not recursive on $S_{g}^{p}$. 
Then, for each component $S''$ of $S_{g}^{p} \setminus \mathrm{Int}N(\delta_2)$, the number of boundary components of $S''$ that are derived from $\delta_2$ (and a boundary component of $S_{g}^{p}$) is one. 
Hence, since $\partial \delta'_2 \subset \delta_2$, it follows that $\partial \delta'_2$ is contained in a single boundary component. 
Therefore $\delta'_2$ is recursive on $S'$. 

(5) Suppose that $\mathcal{T}$ satisfies condition ($*$). 
Pick a boundary component $C$ of $S'$. 
Then, the following holds. 
\begin{align*}
& C \mbox{ has an endpoint of } \delta'_2. \\
& \Leftrightarrow \ C \mbox{ is derived from } \delta_2 \mbox{ and the boundary component(s) of } S_{g}^{p}. 
\end{align*}
Condition ($*$) implies that $C$ is not derived from $\delta_2$ or the boundary component(s) of $S_{g}^{p}$ if $C$ contains an endpoint of $\delta_1$. 
Hence, either $\delta_1 \cap C = \emptyset$ or $\delta'_2 \cap C = \emptyset$ holds. 
\end{proof}

Now, let us compute the maximal lengths of chained triples on holed spheres. 

\begin{lemma}
Suppose that $(\delta_{1}, ( \alpha_1, \ldots, \alpha_m ), \delta_{2})$ is a chained triple on $S_{0}^{3}$. 
Then $m=0$. 
 \label{triple_03}
\end{lemma}
\begin{proof}
There is no essential simple closed curve in $S_{0}^{3}$. 
Hence, $m$ must be $0$. 
\end{proof}

\begin{lemma}
Let $p$ be an integer $\geq 4$. 
Suppose that $(\delta_{1}, (\alpha_1, \ldots, \alpha_m), \delta_{2})$ is a chained triple on $S_{0}^{p}$. 
Then $m \leq p-2$. 
 \label{triple_0p}
\end{lemma}
\begin{proof}
Note that $((\alpha_1, \ldots, \alpha_m), \delta_{2})$ is a ``chained pair" on $S_{0}^{p}$ (a chained pair is defined in \cite[Section 2]{Katayama--Kuno18}). 
By \cite[Lemmas 2.5 and 2.6]{Katayama--Kuno18}, we have $m \leq p-2$. 
\end{proof}

Next, we compute the maximal lengths of chained triples satisfying condition ($*$) on spheres with holes in Lemmas \ref{pre_induction}--\ref{sep_triple_0p} below. 

\begin{lemma}
Suppose that $2 \leq p \leq 3$. 
Then there is no chained triple satisfying condition ($*$) on $S_{0}^{p}$. 
 \label{pre_induction}
\end{lemma}
\begin{proof}
We will prove this lemma in the case where $p=3$. 
Let $\delta_1$ and $\delta_2$ be mutually non-isotopic disjoint arcs on $S_{0}^{3}$. 
If $\partial \delta_1$ is contained in a single boundary component, then $\delta_2$ cannot be essential. 
Hence, $\delta_1$ connects two boundary components. 
Similarly, $\delta_2$ connects two boundary components. 
Since $S_{0}^{3} \setminus \mathrm{Int}N(\delta_1)$ is an annulus, a boundary component of $S_{0}^{3}$ intersects both $\partial \delta_1$ and $\partial \delta_2$. 
Hence, no chained triple satisfies condition ($*$) on $S_{0}^{3}$. 
\end{proof}

\begin{lemma}
Suppose that $(\delta_{1}, ( \alpha_1, \ldots, \alpha_m ), \delta_{2})$ is a chained triple satisfying condition ($*$) on $S_{0}^{4}$. 
Then $m \leq 1$. 
 \label{induction_start}
\end{lemma}
\begin{proof}
Suppose, on the contrary, that $m \geq 2$. 
Then, a component $S'$ of $S_{0}^{4} \setminus \mathrm{Int}N(\delta_{b})$ contains $\delta_{1}$. 
Pick a connected component $\delta'_2$ of $\alpha_m \cap S'$ intersecting $\alpha_{m-1}$. 
Since $m \geq 2$, the arc $\delta_1$ is disjoint from $\delta'_2$. 
By Lemma \ref{key_property_of_chained_triples} (5), it follows that $\delta_1$ is disjoint from $\delta'_2$ and is not isotopic to $\delta'_2$. 
This implies that $S'$ has $(\delta_1, \emptyset, \delta'_2)$, a chained triple of length $0$ satisfying condition $(*)$.  
This is absurd because $S'$ is homeomorphic to $S_{0}^{2}$ or $S_{0}^{3}$ and does not contain such a chained triple by Lemma \ref{pre_induction}. 
Thus, we have $m \leq 1$. 
\end{proof}

\begin{lemma}
Let $p$ be an integer $\geq 5$. 
Suppose that $(\delta_{1}, ( \alpha_1, \ldots, \alpha_m ), \delta_{2})$ is a chained triple satisfying condition ($*$) on $S_{0}^{p}$. 
Then $m \leq p-3$. 
 \label{sep_triple_0p}
\end{lemma}
\begin{proof}
{\bf Case of $p=5$}. 
Suppose, on the contrary, that $m \geq 3$. 
Then, by Lemma \ref{key_property_of_chained_triples} (2) and (5), a connected component $S'$ of $S_{0}^{5} \setminus \mathrm{Int}N(\delta_2)$ contains a chained triple $\mathcal{T}' = (\delta_{1}, (\alpha_1, \ldots, \alpha_{m-1} ), \delta'_{2})$ satisfying condition ($*$). 
Hence, $S'$ contains a length $m-1 \geq 2$ chained triple satisfying condition ($*$). 
This is a contradiction by Lemmas \ref{cutting_surfaces_arc}, \label{pre_induction} and \ref{induction_start}. 
Thus, we have $m \leq 2$. 

{\bf Case of $p \geq 6$}. 
We may assume that $m \geq 3$. 
A similar argument to the $p=5$ case, together with the induction hypothesis, completes the proof. 
\end{proof}

Next, we compute the maximal length of chained triples on a one-holed torus. 

\begin{lemma}
Suppose that $(\delta_{1}, \{ \alpha_1, \ldots, \alpha_m \}, \delta_{2})$ is a chained triple on $S_{1}^{1}$. 
Then $m \leq 2$. 
 \label{triple_11}
\end{lemma}
\begin{proof}
Annuli do not have a chained triple of length $\geq 0$. 
So Lemma \ref{key_property_of_chained_triples} (2) immediately implies the assertion. 
\end{proof}

Next, we compute the maximal lengths of a few different chained triples on a two-holed torus. 

\begin{lemma}
Suppose that $(\delta_{1}, ( \alpha_1, \ldots, \alpha_m ), \delta_{2})$ is a chained triple on $S_{1}^{2}$ satisfying condition $(*)$. 
Then $m \leq 2$. 
 \label{star_trple_12}
\end{lemma}
\begin{proof}
Now, let us assume that $m \geq 3$ and derive a contradiction. 
It follows that either $\delta_1$ or $\delta_2$ is non-separating. 
Since $\delta_2$ is non-separating, $S'=S_{1}^{2} \setminus \mathrm{Int}N(\delta_2)$ is homeomorphic to $S_{0}^{3}$. 
Since $m \geq 3$, $S'$ contains a chained triple of length $m-1$ by Lemma \ref{key_property_of_chained_triples} (2). 
By Lemma \ref{triple_03}, we have $m-1 \leq 0$, which is a contradiction. 
Thus, $m \leq 2$. 
\end{proof}

\begin{lemma}
Suppose that $(\delta_{1}, (\alpha_1, \ldots, \alpha_m ), \delta_{2})$ is a chained triple on $S_{1}^{2}$ such that $\delta_2$ is recursive. 
Then $m \leq 2$. 
 \label{R_trple_12}
\end{lemma}
\begin{proof} 
To the contrary, suppose that $m \geq 3$. 
Since $\delta_2$ is recursive, $\delta_1$ must be non-separating on $S_{1}^{2}$. 
Then $S_{0}^{3} \cong S_{1}^{2} \setminus \mathrm{Int}N(\delta_1)$ contains a chained triple of length $m-1 \geq 2$ by Lemma \ref{key_property_of_chained_triples} (2), this is a contradiction. 
Thus we have $m \leq 2$. 
\end{proof}

Now, let us compute the maximal lengths of chained triples on holed tori. 

\begin{lemma}
Let $p$ be an integer $\geq 3$. 
Suppose that $(\delta_{1}, (\alpha_1, \ldots, \alpha_m ), \delta_{2})$ is a chained triple on $S_{1}^{p}$. 
Then $m \leq p$. 
 \label{triple_1p}
\end{lemma}
\begin{proof}
Proof by induction on $p$. 

{\bf Case of $p = 3$}. 
Assuming $m \geq p = 3$, we will prove that $m = 3$. 
First, we treat the case where $\delta_2$ is non-separating. 
Since $m \geq 3$, the surface $S' = S_{1}^{3} \setminus \mathrm{Int}N(\delta_{2})$ contains a chained triple of length $m-1$ by Lemma \ref{key_property_of_chained_triples} (2), and is homeomorphic to $S_{0}^{p'}$, where $p' \leq 4$.  
Hence, by Lemmas \ref{triple_03} and \ref{triple_0p} we have $m-1 \leq 2$. 
Therefore $m = 3$. 
Next, we treat the case where $\delta_2$ is separating. 
By Lemma \ref{key_property_of_chained_triples} (2) and (4), there is a connected component $S'$ of $S_{1}^{3} \setminus \mathrm{Int}N(\delta_{2})$ that contains a chained triple $\mathcal{T}' = (\delta_{1}, ( \alpha_1, \ldots, \alpha_{m-1} ), \delta'_{2})$ such that $\delta'_2$ is recursive on $S'$. 
Since a sphere with at most three boundary components does not have a chained triple of positive length, $S'$ must be a torus with at most two boundary components. 
Hence, Lemma \ref{R_trple_12} implies $m-1 \leq 2$ i.e., $m \leq 3$. 

{\bf Case of $p \geq 4$}. 
We may assume that $m \geq p > 3$. 
First, we treat the case where $\delta_2$ is non-separating. 
Since $m \geq 3$, the surface $S' = S_{1}^{p} \setminus \mathrm{Int}N(\delta_{2})$ contains a chained triple of length $m-1$ by Lemma \ref{key_property_of_chained_triples} (2), and is homeomorphic to $S_{0}^{p'}$, where $p' \leq p+1$. 
Hence, by Lemmas \ref{triple_03} and \ref{triple_0p} we have $m-1 \leq (p+1)-2$, namely, $m \leq p$. 
Next, we treat the case where $\delta_2$ is separating. 
Then, a connected component $S'$ of $S_{1}^{p} \setminus \mathrm{Int}N(\delta_{2})$ contains a chained triple of length $m-1$ by Lemma \ref{key_property_of_chained_triples} (2). 
When $S'$ is a sphere with a boundary, an argument similar to the case where $\delta_2$ is non-separating implies $m \leq p - 1$. 
When $S'$ is a torus with a boundary, the induction hypothesis implies $m \leq p$. 
\end{proof}

\begin{remark}
Hyperbolic extensions of surfaces preserves chained triples. 
Hence, Lemma \ref{triple_1p} also implies that the maximal length of chained triples on $S_{1}^{2}$ is at most $3$. 
 \label{hyp_ext_ch_tr}
\end{remark}

Now, let us compute the maximal lengths of chained triples on higher genera surfaces. 

\begin{lemma}
Let $g$ be an integer $\geq 2$ and $p$ be an integer $\geq 1$. 
Suppose that $(\delta_{1}, ( \alpha_1, \ldots, \alpha_m ), \delta_{2})$ is a chained triple on $S_{g}^{p}$. 
Then \begin{eqnarray*}
m \leq \left\{ \begin{array}{ll}
2g+p-1 & (p \leq 2) \\
2g+p-2 & (p \geq 3) \\
\end{array} \right.
\end{eqnarray*}
Furthermore, if $p=2$ and $\delta_2$ is recursive, then $m \leq 2g$. 
 \label{triple_gp}
\end{lemma}
\begin{proof}
We give a proof by induction. 
Assume that $m \leq 3$. 

{\bf Case of $(g, p)=(2, 1)$}. 
In the case where $\delta_2$ is non-separating, $S_{2}^{1} \setminus \mathrm{Int}N(\delta_2)$ is homeomorphic to $S_{1}^{2}$ and contains a chained triple of length $m-1$ by Lemma \ref{key_property_of_chained_triples} (2). 
By Remark \ref{hyp_ext_ch_tr}, we have $m-1 \leq 3$ i.e., $m \leq 4 = 2g + p - 1$. 
In the case where $\delta_2$ is separating, a connected component of $S_{2}^{1} \setminus \mathrm{Int}N(\delta_2)$ is homeomorphic to $S_{1}^{1}$ and contains a chained triple of length $m-1$. 
By Lemma \ref{triple_11}, we have $m-1 \leq 2$ i.e., $m \leq 3 < 2g+p-1$. 

{\bf Case of $g \geq 2$, $p=2$ and $\delta_2$ is recursive}. 
In order to prove the assertion we further divide the proof into two cases: (i) $\delta_2$ is non-separating and (ii) $\delta_2$ is separating. 

(i) Suppose that $\delta_2$ is non-separating. 
Then $S_{g}^{2} \setminus \mathrm{Int}N(\delta_2)$ is homeomorphic to $S_{g-1}^{3}$ and contains a chained triple of length $m-1$ by Lemma \ref{key_property_of_chained_triples} (2). 
By Lemma \ref{triple_1p} and the induction hypothesis, we have $m - 1 \leq 2(g-1) + 3 -2$ i.e., $m \leq 2g$. 

(ii) Suppose that $\delta_2$ is separating. 
We claim that either $\delta_1$ or $\delta_2$ does not bound an annulus. 
To see this, assume that $\delta_2$ bounds an annulus, i.e., $S_{g}^{2} \setminus \mathrm{Int}N(\delta_2) \cong S_{g}^{1} \cup S_{0}^{2}$. 
Since any arc that is not recursive and does not intersect $\delta_2$ should be contained in the annulus and since $\alpha_m$ must intersect such an arc, the inequality $m \geq 3$ implies that $\delta_1$ is not such an arc; $\delta_1$ should be recursive.  
Moreover, by (i) we may assume that $\delta_1$ is separating. 
Then $\delta_1$ does not bound an annulus, because $\delta_1$ and $\delta_2$ are not isotopic and that $\delta_1$ is boundary parallel on $S_{g}^{2} \setminus \mathrm{Int}N(\delta_2)$ $(\cong S_{g}^{1})$. 
Hence, either $\delta_1$ or $\delta_2$ does not bound an annulus. 
So we may assume that $\delta_2$ does not bound an annulus. 
Pick a connected component $S'$ of $S_{g}^{2} \setminus \mathrm{Int}N(\delta_2)$ which contains a chained triple $(\delta_1, (\alpha_1, \ldots, \alpha_{m-1}), \delta_2')$ obtained from Lemma \ref{key_property_of_chained_triples} (2). 
Since $\delta_2$ does not bound an annulus, the component $S'$ is homeomorphic to a surface of genus $\leq g-1$ with at most two boundary components. 
Moreover, Lemma \ref{key_property_of_chained_triples} (2) implies that the arc $\delta'_2$ derived from $\alpha_m$ is recursive, because $\delta_2$ is separating. 
In the case where the number of boundary components of $S'$ is $2$, by Lemma \ref{R_trple_12} and the induction hypothesis we have $m -1 \leq 2(g-1)$, i.e., $m \leq 2g - 1$. 
Since a hyperbolic extension of a surface preserves chained triple, the case where the number of boundary components of $S'$ is $1$ can be reduced to the case where the number of boundary components of $S'$ is $2$. 
Thus we have $m \leq 2g-1$. 

{\bf Case of $(g,p)=(2,2)$}. 
By Lemma \ref{key_property_of_chained_triples} (2), a connected component $S'$ of $S_2^2 \setminus \mathrm{Int}N(\delta_2)$ contains a chained triple of length $m-1$. 
Then $S'$ is homeomorphic to $S_{2}^{1}$; otherwise the genus of $S'$ is at most $1$ and the number of boundary components of $S'$ is at most $3$. 
In both cases we have $m -1 \leq 4$ by Lemmas \ref{triple_03}, \ref{triple_1p} and the assertion in the case where $(g,p)=(2, 1)$. 
Hence $m \leq 5 = 2g+p-1$. 

{\bf Case of $g \geq 3$ and $p=1$}. 
In the case where $\delta_2$ is non-separating, $S_{g}^{1} \setminus \mathrm{Int}N(\delta_2)$ is homeomorphic to $S_{g-1}^{2}$ and contains a chained triple of length $m-1$. 
By the induction hypothesis, we have $m-1 \leq 2(g-1)+2-1$ i.e., $m \leq 2g = 2g + p -1$.  
In the case where $\delta_2$ is separating, a connected component of $S_{g}^{1} \setminus \mathrm{Int}N(\delta_2)$ is homeomorphic to $S_{g'}^{1}$ ($g' \leq g-1$) and contains a chained triple of length $m-1$. 
The induction hypothesis gives $m-1 \leq 2(g-1)$, i.e., $m \leq 2g-1 < 2g+p-1$. 

{\bf Case of $g \geq 3$ and $p=2$}. 
A connected component $S'$ of $S_g^2 \setminus \mathrm{Int}N(\delta_2)$ contains a chained triple of length $m-1$. 
Then either $S'$ is homeomorphic to $S_{g}^{1}$ or the genus of $S'$ is at most $g-1$ and the number of boundary components of $S'$ is at most $3$. 
By the induction hypothesis, we have $m -1 \leq 2g$. 
Hence, $m \leq 2g+1 = 2g+p-1$. 

{\bf Case of $g \geq 2$ and $p \geq 3$}. 
A connected component $S'$ of $S_g^p \setminus \mathrm{Int}N(\delta_2)$ contains a chained triple of $(\delta_1, (\alpha_1, \ldots, \alpha_{m-1}), \delta_2')$. 
Suppose that $\delta_2$ is separating. 
Then $\delta_2'$ is recursive. 
When the genus of $S'$ is $g$, the induction hypothesis gives $m-1 \leq 2g + (p-1) -2$, i.e., $m \leq 2g + p - 2$. 
When the genus of $S'$ is at most $g-1$, the induction hypothesis gives $m - 1 \leq 2(g-1) + (p+1) - 2$, i.e., $m \leq 2g + p -2$. 
Next, let us consider the case where $\delta_2$ is non-separating. 
Then, $S'$ is homeomorphic to $S_{g-1}^{p+1}$. 
The induction hypothesis gives $m-1 \leq 2(g-1)+(p+1)-2$. 
Thus, the assertion holds. 
\end{proof}

We will also use the following sequence on surfaces. 

\begin{definition}
A quadruple $(\delta_1, (\alpha_1, \ldots, \alpha_m), \delta_2, \delta_3)$ of arcs $\delta_1$, $ \delta_2$ and $\delta_3$, and a sequence $(\alpha_1, \ldots, \alpha_m)$ of closed curves on a surface $S$ is said to be {\it $Y$-chained} if the quadruple satisfies the following conditions: 
\begin{enumerate}
 \item[$\bullet$] $(\delta_1, (\alpha_1, \ldots, \alpha_m), \delta_2)$ is a chained triple on $S$, 
 \item[$\bullet$] $(\delta_1, (\alpha_1, \ldots, \alpha_m), \delta_3)$ is a chained triple on $S$, 
 \item[$\bullet$] $\delta_2, \delta_3$ are disjoint and non-isotopic, and 
 \item[$\bullet$] there is a boundary component $C$ of $S$ such that $\partial \delta_2 \cup \partial \delta_3 \subset C$. 
\end{enumerate}
\end{definition}

We prepare a lemma for $Y$-chained quadruples, which is similar to Lemma \ref{key_property_of_chained_triples}. 

\begin{lemma}
Suppose $m \geq 3$. 
Suppose also that $(\delta_1, (\alpha_1, \ldots, \alpha_m), \delta_2, \delta_3)$ is a $Y$-chained quadruple on a surface $S$. 
Then, there is an arc $\delta_1'$ on $S \setminus \mathrm{Int}N(\delta_1)$ derived from $\alpha_1$ such that $(\delta_1', (\alpha_2, \ldots, \alpha_m), \delta_2, \delta_3)$ is $Y$-chained quadruple on $S \setminus \mathrm{Int}N(\delta_1)$. 
 \label{Y_lemma}
\end{lemma}
\begin{proof}
Suppose that $(\delta_1, (\alpha_1, \ldots, \alpha_m), \delta_2, \delta_3)$ is a $Y$-chained quadruple on a surface $S$. 
Pick a proper arc $\delta_1'$ on $S \setminus \mathrm{Int}N(\delta_1)$ which is derived from $\alpha_1$ such that $\delta_1' \cap \alpha_2 = \emptyset$. 
By \cite{Katayama--Kuno18}, both $(\delta_1, (\alpha_2, \ldots, \alpha_m), \delta_2)$ and $(\delta_1, (\alpha_2, \ldots, \alpha_m), \delta_3)$ are chained triples on $S \setminus \mathrm{Int}N(\delta_1)$. 
Obviously, $\delta_2$ and $\delta_3$ are disjoint on $S \setminus \mathrm{Int}N(\delta_1)$. 
Since any isotopy between $\delta_2$ and $\delta_3$ on $S \setminus \mathrm{Int}N(\delta_1)$ gives rise to an isotopy between $\delta_2$ and $\delta_3$ on $S $ relative to $\mathrm{Int}N(\delta_1)$, it holds that $\delta_2$ and $\delta_3$ are non-isotopic on $S \setminus \mathrm{Int}N(\delta_1)$. 
Since the boundary component of $S$ containing $\partial \delta_2 \cup \partial \delta_3$ naturally becomes a boundary component of $S \setminus \mathrm{Int}N(\delta_1)$, $(\delta_1', (\alpha_2, \ldots, \alpha_m), \delta_2, \delta_3)$ is a $Y$-chained quadruple on $S \setminus \mathrm{Int}N(\delta_1)$. 
\end{proof}

Let us compute the maximal lengths of $Y$-chained quadruples on surfaces. 
First, we treat holed spheres. 

\begin{lemma}
Any two disjoint arcs $\delta_2, \delta_3$ such that $\partial\delta_2 \cup \partial \delta_3$ is contained in a single boundary component of $S_0^3$ are isotopic. 
 \label{Y_03}
\end{lemma}
\begin{proof}
Since $S_0^3 \setminus \mathrm{Int}N(\delta_3)$ consists of two copies of annuli, $\delta_2$ is boundary parallel on $S_0^3 \setminus \mathrm{Int}N(\delta_3)$. 
This implies that $\delta_2$ and $\delta_3$ are isotopic on $S_0^3$. 
\end{proof}

\begin{lemma}
Suppose that $(\delta_1, (\alpha_1, \ldots, \alpha_m), \delta_2, \delta_3)$ is a $Y$-chained quadruple on $S_0^4$. 
Then $m \leq 1$. 
 \label{Y_04}
\end{lemma}
\begin{proof}
Suppose, on the contrary, that $m \geq 2$. 
Then a connected component $S'$ of $S_{0}^4 \setminus \mathrm{Int}N(\alpha_1)$ contains $\delta_2$ and $\delta_3$ and is homeomorphic to $S_{0}^{3}$. 
By Lemma \ref{Y_03}, there is an isotopy between $\delta_2$ and $\delta_3$ on $S'$. 
However, this isotopy in turn gives an isotopy between $\delta_2$ and $\delta_3$ on the original surface $S_{0}^4$, which is a contradiction. 
\end{proof}

\begin{lemma}
Let $p$ be an integer $\geq 5$. 
Suppose that $(\delta_1, (\alpha_1, \ldots, \alpha_m), \delta_2, \delta_3)$ is a $Y$-chained quadruple on $S_0^p$. 
Then $m \leq p-3$. 
 \label{Y_0p}
\end{lemma}
\begin{proof}
Proof by induction on $p$. 

{\bf Case of $p=5$}. 
Suppose, on the contrary, that $m \geq 3 > p-3$. 
Then, a connected component $S'$ of $S \setminus \mathrm{Int}N(\alpha_1)$ contains a length $m-1$ chained quadruple by Lemma \ref{Y_lemma}. 
Since $S'$ is homeomorphic to a sphere with at most four boundary components, by Lemmas \ref{Y_03} and \ref{Y_04} $m-1 \geq 2$ is impossible. 
This is a contradiction, so $m \leq 2$. 

{\bf Case of $p \geq 6$}. 
An argument similar to the $p=5$ case, together with the induction hypothesis, completes the proof. 
\end{proof}

Next, let us consider holed tori. 

\begin{lemma}
No $Y$-chained quadruple is contained in $S_{1}^{1}$. 
 \label{Y_11}
\end{lemma}
\begin{proof}
Any triple of mutually non-isotopic arcs on $S_{1}^{1}$ intersects minimally non-trivially.  
\end{proof}

\begin{lemma}
Suppose that $(\delta_1, (\alpha_1, \ldots, \alpha_m), \delta_2, \delta_3)$ is a $Y$-chained quadruple on $S_1^2$. 
Then $m \leq 2$. 
Furthermore, if $(\delta_1, (\alpha_1, \alpha_2), \delta_2, \delta_3)$ is a $Y$-chained quadruple on $S_{1}^{2}$, then $\delta_2$ is separating on $S_{1}^{2}$. 
 \label{Y_12}
\end{lemma}
\begin{proof}
First, we treat the case where $\delta_2$ is separating. 
In this case $S_1^2 \setminus \mathrm{Int}N(\delta_2)$ is homeomorphic to a disjoint union of a one-holed torus $T$ and an annulus $A$. 
If $\delta_1$ is contained in $A$, then any essential simple closed curve on $S_{1}^2$ intersecting $\delta_1$ must also intersect $\delta_2$, so $m \leq 1$. 
Thus, we may assume that $\delta_1$ is contained in $T$. 
The arc $\delta_3$ is contained in either $T$ or $A$, because $\delta_3$ is disjoint from $\delta_2$. 
Since $\partial \delta_2$ and $\partial \delta_3$ are contained in the same boundary component and since $\delta_3$ is not isotopic to $\delta_2$, the arc $\delta_3$ must be contained in $T$. 
Since $\delta_1$ is not isotopic to $\delta_3$, it follows that any essential simple closed curve in $T$ must intersect either $\delta_1$ or $\delta_3$. 
This implies that $m \leq 2$ (if $m \geq 2$, $\alpha_2$ must intersect $\delta_3$). 

Next, we treat the case where $\delta_2$ is not separating. 
We may assume that $\delta_3$ is not separating either. 
Then $S':=S_{1}^{2} \setminus \mathrm{Int}N(\delta_2)$ is homeomorphic to $S_{0}^{3}$ and $\delta_3$ connects the two boundary components $C_1$ and $C_2$ which are derived from $\delta_2$ on $S'$. 
Any essential simple closed curve on $S_{1}^{2}$ that is disjoint from $\delta_2$ is isotopic to either $C_1$ or $C_2$ on $S'$ and thereby intersecting $\delta_3$ on the original surface $S_{1}^{2}$. 
Hence, we have $m \leq 1$. 

The last assertion of this lemma immediately follows from the above. 
\end{proof}

The maximal lengths of $Y$-chained quadruples on higher genera surfaces can be computed as follows. 

\begin{lemma}
Suppose that $(\delta_1, (\alpha_1, \ldots, \alpha_m), \delta_2, \delta_3)$ is a $Y$-chained quadruple on $S_g^p$.  
If the pair $(g, p)$ satisfies either $g = 1$ and $p \geq 3$ or $g \geq 2$ and $p \geq 2$, then $m \leq 2g + p -3$. 
 \label{Y_general}
\end{lemma}
\begin{proof}
Proof by induction on $(g, p)$. 

{\bf Case of $(g, p) = (1, 3)$}. 
In the case where $\delta_1$ is non-separating, $S_{1}^{3} \setminus \mathrm{Int}N(\delta_1)$ is homeomorphic to $S_{0}^{4}$, so we have $m \leq 2$ by Lemma \ref{Y_04}. 
Next, let us consider the case where $\delta_1$ is separating. 
Suppose, on the contrary, that $m \geq 3$. 
By Lemma \ref{Y_lemma}, $S_{g}^{p} \setminus \mathrm{Int}N(\delta_1)$ contains a $Y$-chained quadruple $L$ of length $m-1 \geq 2$. 
By Lemmas \ref{Y_04} and \ref{Y_11}, $S \setminus \mathrm{Int}N(\delta_1)$ must be homeomorphic to a disjoint union of a two-holed torus $T$ and an annulus $A$, and $L$ is contained in $T$. 
Lemma \ref{Y_12} implies that $m$ must be equal to $3$ and $\delta_2$ is separating on $T$. 
It follows that $\delta_2$ is separating on the original surface $S_{1}^{3}$. 
Moreover, $\delta_3$ is non-separating on $S_{1}^{3}$. 
Therefore $S_{1}^{3} \setminus \mathrm{Int}N(\delta_3)$ is homeomorphic to $S_{0}^{4}$. 
Obviously, $\alpha_2$ and $\delta_2$ are essential on $S_{1}^{3} \setminus \mathrm{Int}N(\delta_3)$. 
Every essential simple closed curve on $S_{1}^{3} \setminus \mathrm{Int}N(\delta_3)$ intersecting $\alpha_2$ but disjoint from $\delta_2$ forms a bigon with $\alpha_2$. 
Since $\alpha_1$ does not form a bigon with $\alpha_2$ on $S_{1}^{3}$, $\alpha_3$ is isotopic to a boundary component $C$ of $S_{1}^{3} \setminus \mathrm{Int}N(\delta_3)$. 
However, this implies that $\alpha_1$ intersects $\alpha_2$ trivially on $S_{1}^{3}$, which is a contradiction. 
Hence, we have $m \leq 2$. 

In the following cases we may assume that $m \geq 3$. 
By Lemma \ref{Y_lemma}, $S_{g}^{p} \setminus \mathrm{Int}N(\delta_1)$ contains a $Y$-chained quadruple of length $m-1$. 

{\bf Case of $g=1$ and $p \geq 4$}. 
Lemma \ref{Y_0p} and the induction hypothesis give $m-1 \leq p - 3$. 

{\bf Case of $g \geq 2$ and $p \geq 1$}. 
Lemma \ref{Y_0p} and the induction hypothesis give $m-1 \leq 2(g-1) + (p+1) - 3$, i.e., $m \leq 2g + p - 3$. 
\end{proof}

We are now ready to decide whether there is a cyclic chain of a given length on surfaces. 
To begin with, we compute the maximal lengths of cyclic chains whose two consecutive curves intersect exactly once on surfaces. 

\begin{lemma}
Let $(\alpha_1, \ldots, \alpha_m)$ be a cyclic chain on $S_g^p$. 
If $\mathrm{int}(\alpha_i, \alpha_{i+1}) = 1$ $(1 \leq i \leq m)$, then $m \leq 2g+2$. 
 \label{int_1_case}
\end{lemma}
\begin{proof}
We will prove that the genus $g'$ of a subsurface $N(\alpha_1 \cup \cdots \cup \alpha_m)$ is at least $\frac{m}{2} - 1$. 
This fact directly implies that $m \leq 2g+2$, because the genus of a subsurface does not exceed that of the ambient surface. 
First, we show that the number $n$ of boundary components of $N(\alpha_1 \cup \cdots \cup \alpha_m)$ is either at most $3$ or at most $4$ according to whether $m$ is odd or even. 
Consider the case where $m$ is odd. 
The number of the boundary components of a sub-subsurface $N(\alpha_1 \cup \cdots \cup \alpha_{m-1})$ is $1$. 
Note that the subsurface $N(\alpha_1 \cup \cdots \cup \alpha_m)$ can be obtained from $(\alpha_1, \ldots, \alpha_{m-1})$ by attaching two bands to its two boundary components so that the bands connect the boundary components. 
Since these increase at most $1$ boundary component, $n$ is at most $3$. 
Next, we consider the case where $m$ is even. 
The number of boundary components of a sub-subsurface $N(\alpha_1 \cup \cdots \cup \alpha_{m-1})$ is $2$. 
Hence, $n$ is at most $4$. 

Using the Euler characteristic we obtain the equality $2- 2g' - n = - m$, i.e., $g' = \frac{2+m-n}{2}$. 
Thus $g' \geq \frac{m}{2} - 1$. 
\end{proof}

Note that a cyclic chain of length $m$ on a surface gives a shorter cyclic chain as follows. 

\begin{lemma}
Let $m$ be an integer greater than $3$, $S$ a surface with negative Euler characteristic. 
If $S$ has a cyclic chain of length $m$, then $S$ has a cyclic chain of length $m-1$. 
\label{chain_shorter}
\end{lemma}
\begin{proof}
Let $(\alpha_1, \ldots, \alpha_m)$ be a cyclic chain on $S$, $T_{\alpha_m}$ be a Dehn twist along a closed curve $\alpha_m$. 
Then $(\alpha_1, \ldots, \alpha_{m-2}, T_{\alpha_m}(\alpha_{m-1}))$ is a cyclic chain of length $m-1$. 
\end{proof}

Let us compute the maximal lengths of cyclic chains on surfaces. 

\begin{lemma}
The complement graph $C_m^c$ of the cyclic graph on $m$ vertices is a full subgraph of $\mathcal{C}(S_{g}^{p})$ if and only if 
\begin{eqnarray*}
m \leq \left\{ \begin{array}{ll}
p & (g = 0, \ p \geq 5) \\
3 & (g, p)=(1, 1) \\
5 & (g, p)=(1, 2) \\
p + 2 & (g=1, \ p \geq 3)  \\
2g + 2 & (g \geq 2, \ p= 0) \\
2g + p + 1 & (g \geq 2, \ 1 \leq p \leq 2). \\
2g + p & (g \geq 2, \ p \geq 3). \\
\end{array} \right.
\end{eqnarray*}
 \label{cyclic_full_subgraph}
\end{lemma}
\begin{proof}
{\bf Case of g = 0 and $p \geq 5$}. 
From \cite[Figure 7]{Katayama--Kuno18}, we can see that $S_{0}^{p}$ contains a cyclic chain of length $p$. 
By Lemma \ref{chain_shorter}, we can obtain full embeddings of shorter cyclic graphs into $\mathcal{C}(S_{0}^{p})$. 
On the other hand, $S_{0}^{p}$ does not contain a linear chain of length $\geq p$ by \cite[Theorem 2.2]{Katayama--Kuno18}. 
Since any cyclic chain of length $m$ induces a linear chain of length $m-1$, we have $C_{m}^c \not \leq \mathcal{C}(S_{0}^{p})$ for each $m \geq p+1$. 

{\bf Case of $(g, p)=(1, 1)$}. 
The left picture in Figure \ref{S11_C3} shows $C_3^c \leq \mathcal{C}(S_{1}^{1})$. 
\begin{figure}
\includegraphics[clip, scale=0.40]{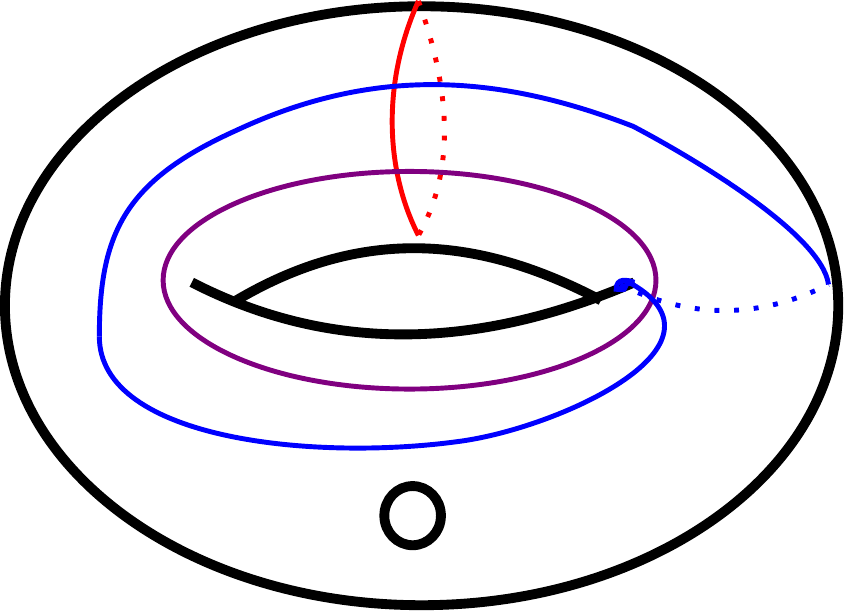}
\hspace{5mm} \includegraphics[clip, scale=0.40]{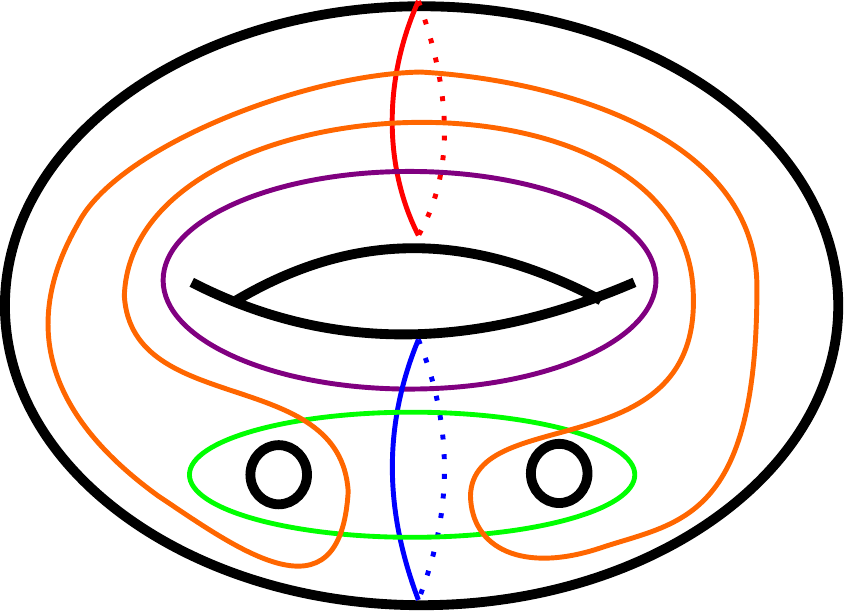}
\caption{Longest cyclic chains on tori with at most two boundary components. \label{S11_C3}}
\end{figure}
Note that a cyclic chain of length $\geq 4$ contains a pair of non-isotopic disjoint closed curves. 
A one-holed torus does not contain such a pair. 

{\bf Case of $(g, p)=(1, 2)$}. 
The right picture in Figure \ref{S11_C3} shows $C_5^c \leq \mathcal{C}(S_{1}^{2})$. 
By Lemma \ref{cahin_shorter}, we can obtain full embeddings of $C_3^c$ and $C_4^c$ into $\mathcal{C}(S_{1}^{2})$. 
A cyclic chain of length $\geq 6$ contains a triple of pairwise non-isotopic disjoint closed curves. 
A two-holed torus does not contain such a triple. 

{\bf Case of $g=1, \ p \geq 3$}. 
Figure \ref{S1p_Cn} shows $C_{p+2}^c \leq \mathcal{C}(S_{1}^{p})$. 
\begin{figure}
\includegraphics[clip, scale=0.40]{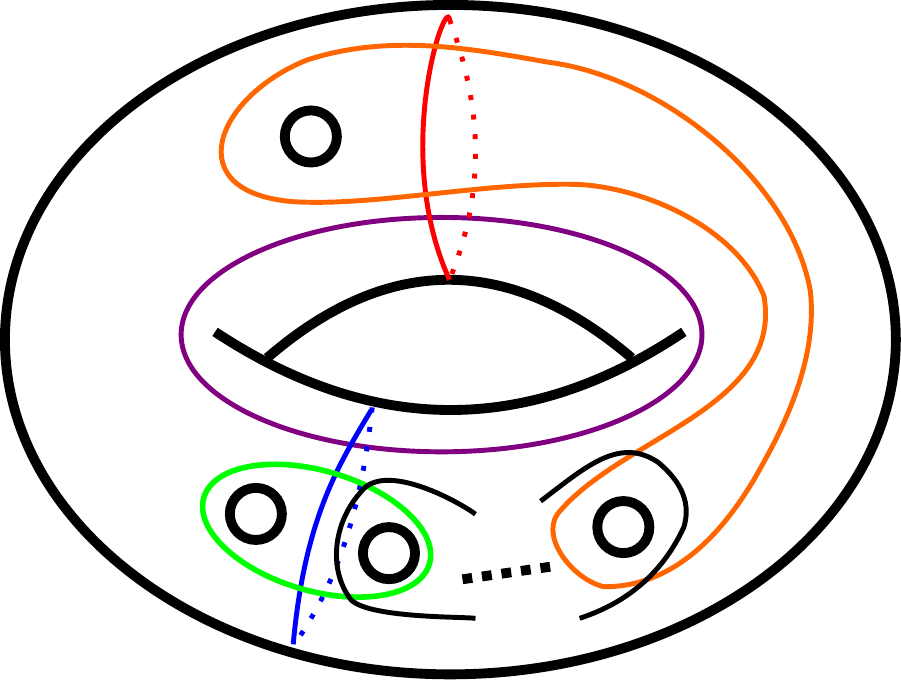}
\caption{Longest cyclic chains on tori with at least three boundary components. \label{S1p_Cn}}
\end{figure}
Suppose that $( \alpha_1, \ldots, \alpha_m )$ is a cyclic chain on $S_{1}^{p}$. 
We may assume that $m \geq p+2 \geq 5$. 
By Lemma \ref{int_1_case}, $( \alpha_1, \ldots, \alpha_m )$ contains a pair of closed curves whose intersection number is greater than $1$. 
Hence, without loss of generality we may assume that the intersection number of $\alpha_{m-1}$ and $\alpha_m$ is greater than $1$. 
Thus, a connected component $S'$ of $S_{1}^{p} \setminus \mathrm{Int}N(\alpha_m)$ contains a quadruple $L := (\delta_1, (\alpha_2, \ldots, \alpha_{m-2}), \delta_2, \delta_3)$. 
The quadruple $L$ is $Y$-chained or not according to whether $\delta_2$ intersects $\alpha_{m-2}$ or not. 

(i) Suppose that $\delta_2$ intersects $\alpha_{m-2}$. 
Then, $L$ is $Y$-chained on $S'$. 
By Lemmas \ref{Y_0p}, \ref{Y_12} and \ref{Y_general} we have $m-3 \leq (p+2) - 3$ (resp.\ $m-3 \leq 2$) if $p \geq 4$ (resp.\ $p = 3$). 

(ii) Suppose that $\delta_2$ is disjoint from $\alpha_2$. 
Then $\delta_2$ is disjoint from a chained triple $(\delta_1, (\alpha_2, \ldots, \alpha_{m-2}), \delta_3)$. 
Since $\delta_2$ is isotopic to neither $\delta_1$ nor $\delta_3$, a connected component of $S_{1}^{p} \setminus \mathrm{Int}(N(\alpha_m) \cup N(\delta_2))$ contains $(\delta_1, (\alpha_2, \ldots, \alpha_{m-2}), \delta_3)$. 
By Lemmas \ref{triple_0p}, \ref{triple_11} and \ref{triple_1p}, we have $m-3 \leq p-1$. 
Thus, we have $m \leq p+2$. 

{\bf Case of $g \geq 2$ and $p = 0$}. 
From \cite[Figure 7]{Katayama--Kuno18}, we can see that $S_{g}$ contains a cyclic chain of length $2g+2$. 
On the other hand, $S_{0}^{p}$ does not contain a linear chain of length $\geq 2g+2$ by \cite[Theorem 2.2]{Katayama--Kuno18}. 
Hence, $C_{m}^c \not \leq \mathcal{C}(S_{0}^{p})$ for each $m \geq 2g+3$. 

{\bf Case of $g \geq 2, \ 1 \leq p \leq 2$}. 
The left picture in Figure \ref{Sg1_Cn} shows $C_{2g+2}^c \leq \mathcal{C}(S_{g}^{1})$ and the right picture shows $C_{2g+3}^c \leq \mathcal{C}(S_{g}^{2})$. 
Suppose that $( \alpha_1, \ldots, \alpha_m )$ is a cyclic chain on $S_{g}^{p}$. 
We may assume that $m \geq 2g+p+1 > 3$. 
Then a connected component $S'$ of $S_{g}^{p} \setminus \mathrm{Int}N(\alpha_m)$ contains a chained triple $(\delta_{a}, \{ \alpha_2, \ldots, \alpha_{m-2} \}, \delta_{b})$. 
Hence, by Lemma \ref{triple_gp} we have $m-3 \leq 2(g-1)+(p+2)-2$ i.e., $m \leq 2g+p+1$. 

{\bf Case of $g \geq 2, \ p \geq 3$}. 
Figure \ref{Sgp_Cn} shows $C_{2g+p}^c \leq \mathcal{C}(S_{g}^{p})$. 
Suppose that $( \alpha_1, \ldots, \alpha_m )$ is a cyclic chain on $S_{g}^{p}$. 
We may assume that $m \geq 2g + p > 2g+2 = 6$. 
By Lemma \ref{int_1_case}, $( \alpha_1, \ldots, \alpha_m )$ contains a pair of closed curves whose intersection number is greater than $1$. 
Hence, without loss of generality we may assume that the intersection number of $\alpha_{m-1}$ and $\alpha_m$ is greater than $1$. 
Then a connected component $S'$ of $S_{g}^{p} \setminus \mathrm{Int}N(\alpha_m)$ contains a quadruple $L := (\delta_1, (\alpha_2, \ldots, \alpha_{m-2}), \delta_2, \delta_3)$. 
The quadruple $L$ is $Y$-chained or not according to whether $\delta_2$ intersects $\alpha_{m-2}$ or not. 

(i) Suppose that $\delta_2$ intersects $\alpha_{m-2}$. 
Then $L$ is $Y$-chained on $S'$. 
By Lemmas \ref{Y_0p} and \ref{Y_general}, we have $m-3 \leq 2(g-1) + (p+2) - 3$, i.e., $m \leq 2g + p$. 

(ii) Suppose that $\delta_2$ is disjoint from $\alpha_2$. 
Then $\delta_2$ is disjoint from a chained triple $(\delta_1, (\alpha_2, \ldots, \alpha_{m-2}),  \delta_3)$. 
Since $\delta_2$ is neither isotopic to $\delta_1$ nor $\delta_3$, a connected component of $S_{g}^{p} \setminus \mathrm{Int}N(\alpha_m) \cup N(\delta_2)$ contains $(\delta_1, (\alpha_2, \ldots, \alpha_{m-2}), \delta_3)$. 
From Lemmas \ref{triple_0p}, \ref{triple_11}, \ref{triple_1p} and \ref{triple_gp}, it holds that $m-3 \leq 2(g-2)+(p+3)-2$, i.e., $m \leq 2g+p$. 
\end{proof}

\begin{figure}
\includegraphics[clip, scale=0.35]{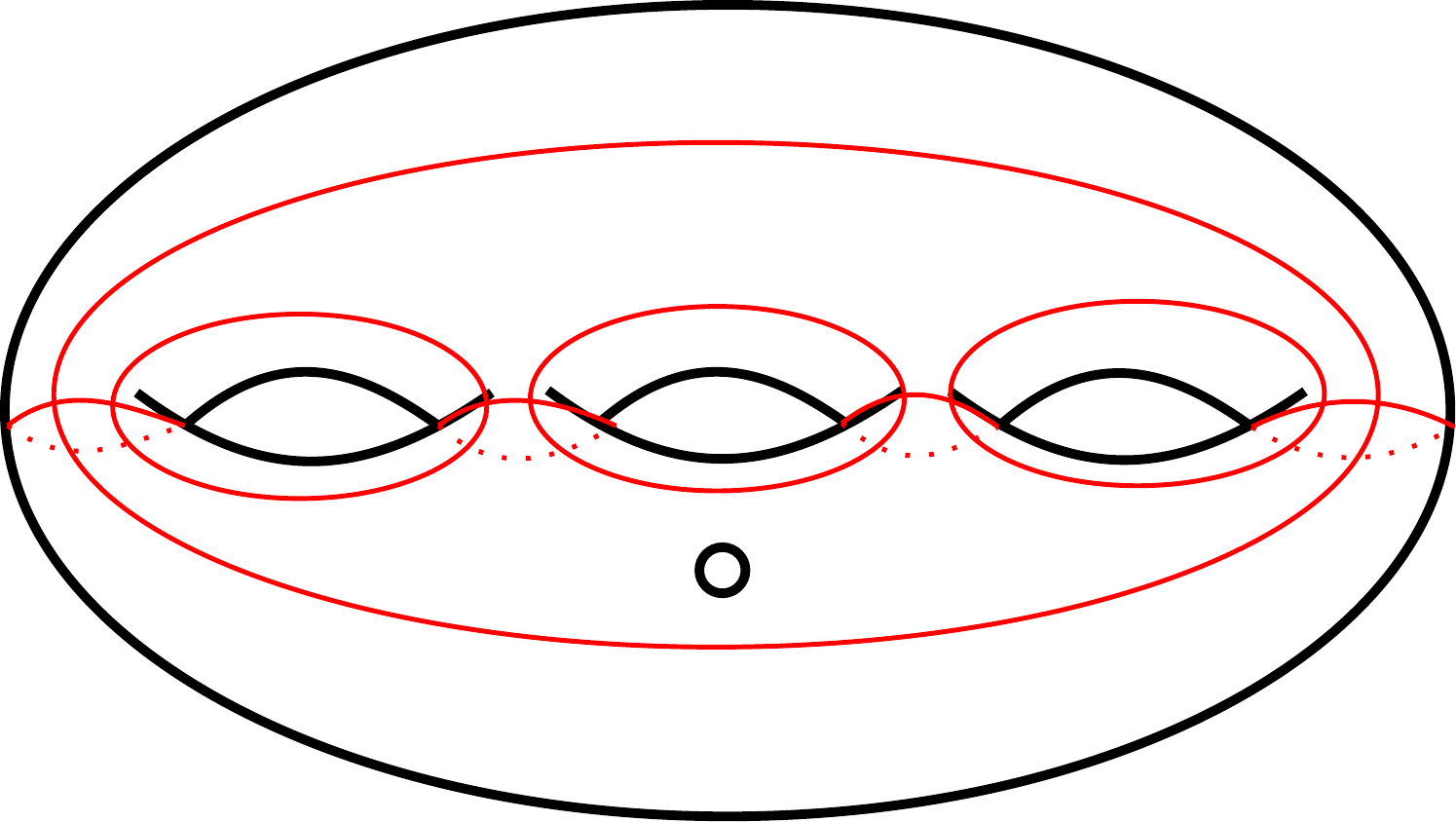}
\hspace{5mm} \includegraphics[clip, scale=0.45]{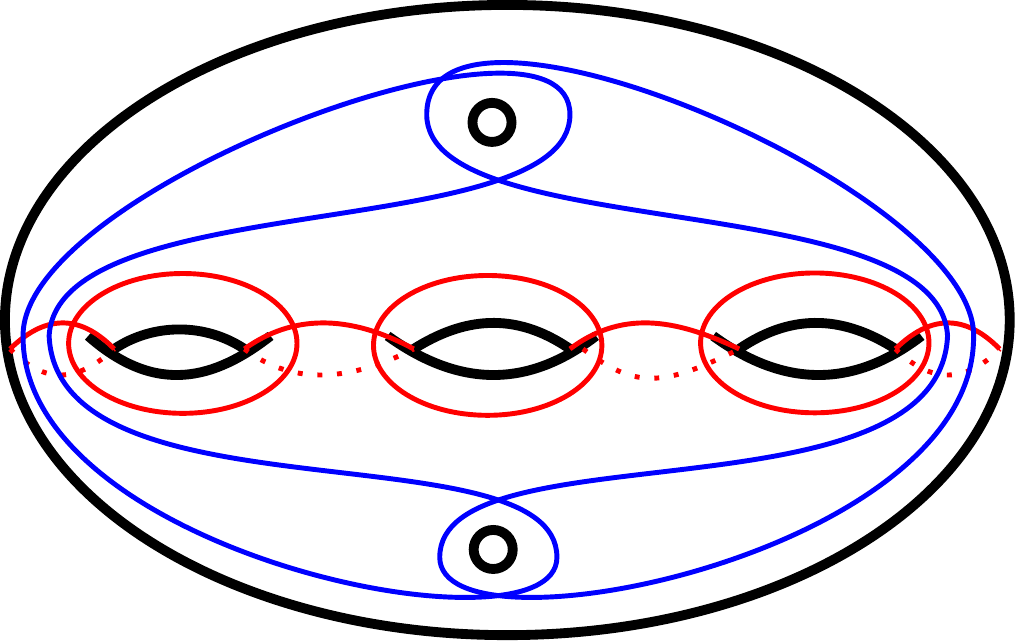}
\caption{Longest cyclic chains on surfaces of genus $\geq 2$ with at most two boundary components. \label{Sg1_Cn}}
\end{figure}

\begin{figure}
\centering
\includegraphics[clip, scale=0.45]{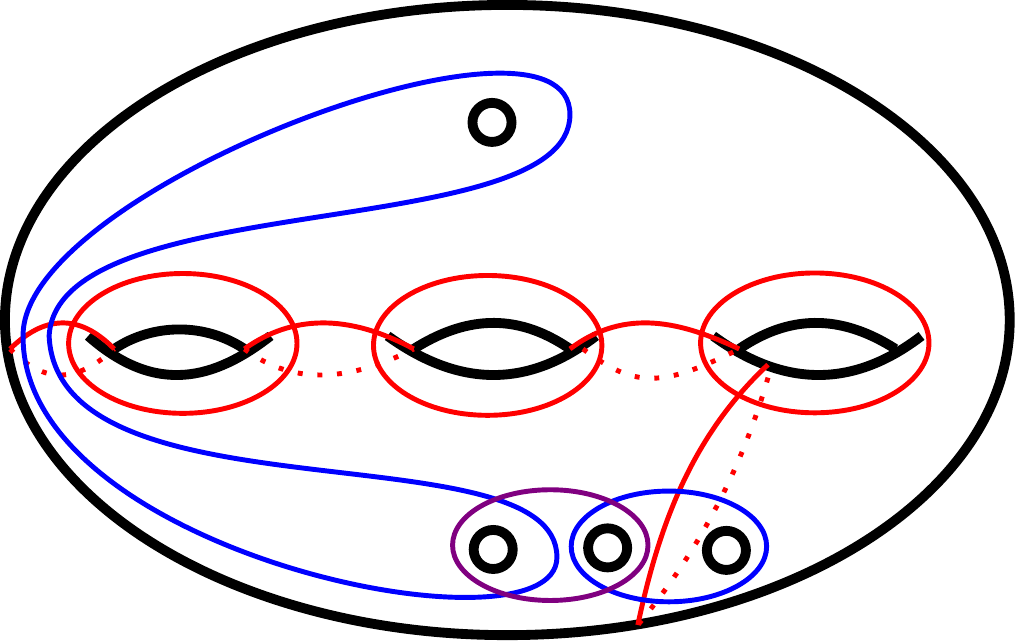}
\caption{Longest cyclic chains on surfaces of genus $\geq 2$ with at least three boundary components.  \label{Sgp_Cn}}
\end{figure}

For a pair of graphs $\Lambda_1$ and $\Lambda_2$, the {\it graph-join} $*_{i=1}^{2} \Lambda_i = \Lambda_1 * \Lambda_2$ is a graph such that the vertex set is a disjoint union of $V(\Lambda_1)$ and $V(\Lambda_2)$, and that the edge set is a disjoint union of $E(\Lambda_1)$, $E(\Lambda_2)$ and $\{ [u, v] \mid u \in V(\Lambda_1), \ v \in V(\Lambda_2) \}$. 

\begin{lemma}
$C_m^c * K_1$ is a full subgraph of $\mathcal{C}(S_{g}^{p})$ if and only if 
\begin{eqnarray*}
m \leq \left\{ \begin{array}{ll}
3 & (g, p) \in \{(0, 5), (1, 2) \} \\
p-1 & (g=0, \ p \geq 6) \\
p + 2 & (g = 1, \ p \geq 3) \\
2g + 1 & (g \geq 2, \ p = 0)  \\
2g + p & (g \geq 2, \ p \geq 1) \\
\end{array} \right.
\end{eqnarray*}
 \label{cyclic_plus_full_subgraph}
\end{lemma}
\begin{proof}
We will reduce the ``if part" of Lemma \ref{cyclic_plus_full_subgraph} to Lemma \ref{cyclic_full_subgraph}. 
Suppose $g \geq 2$ and $p \leq 1$. 
Then, any cyclic chain of length $m$ on $S_{g-1}^{p+2}$ induces a cyclic chain of length $m$ on $S_{g}^{p}$ via an annular extension $S_{g-1}^{p+2} \rightarrow S_{g}^{p}$. 
Hence, the disjoint union of the core of the attached annulus and the induced cyclic chain of length $m$ induces a full embedding $C_{m}^c * K_1 \leq \mathcal{C}(S_{g}^{p})$. 
Consequently, $C_m^c \leq \mathcal{C}(S_{g-1}^{p+2})$ implies $C_m^c * K_1 \leq \mathcal{C}(S_{g}^{p})$.  
Next, we consider the case where $p \geq 2$ (and $g \geq 0$). 
Use a hyperbolic extension $S_{g}^{p-1} \rightarrow S_{g}^{p}$.  
Then $C_{m}^c \leq \mathcal{C}(S_{g}^{p-1})$ implies $C_{m}^{c} * K_1 \leq \mathcal{C}(S_{g}^{p})$. 
Hence, Lemma \ref{cyclic_full_subgraph} together with \cite[Figure 7]{Katayama--Kuno18} completes the ``if part" of Lemma \ref{cyclic_plus_full_subgraph}. 

Now, we prove the ``only if part" of Lemma \ref{cyclic_plus_full_subgraph}. 

{\bf Case of $g=0, \ p \geq 5$}. 
Suppose that $C_{m}^c * K_1 \leq \mathcal{C}(S_{0}^{p})$. 
Then, a cyclic chain of length $m$ is contained in a connected component $S'$ of the surface $S_{0}^{p} \setminus \mathrm{Int}N(\alpha)$, where $\alpha$ is a closed curve corresponding to the graph $K_1$. 
Since $S'$ is a sphere with $\leq p-1$ holes, $m$ must be not more than $p-1$. 

{\bf Case of $(g, p)=(1, 2)$}. 
Suppose that $C_{m}^c * K_1 \leq \mathcal{C}(S_{1}^{2})$. 
Then, a cyclic chain of length $m$ is contained in a connected component $S'$ of the surface $S_{1}^{2} \setminus \mathrm{Int}N(\alpha)$, where $\alpha$ is a closed curve corresponding to the graph $K_1$. 
Since $S'$ is homeomorphic to either $S_{0}^{4}$ or $S_{1}^{1}$, we have $m \leq 3$. 

{\bf Case of $g=1, \ p \geq 3$}. 
Suppose that $C_{m}^c * K_1 \leq \mathcal{C}(S_{1}^{p})$. 
Then, a cyclic chain of length $m$ is contained in a connected component $S'$ of the surface $S_{1}^{2} \setminus \mathrm{Int}N(\alpha)$, where $\alpha$ is a closed curve corresponding to the graph $K_1$. 
The component $S'$ is homeomorphic to either $S_{0}^{p'}$ or $S_{1}^{p''}$, where $p' \leq p+2$ and $p'' \leq p-1$. 
Hence, by Lemma \ref{cyclic_full_subgraph}, we have $m \leq p+2$. 

{\bf Case of $g \geq 2, \ p=0$}. 
Suppose that $C_{m}^c * K_1 \leq \mathcal{C}(S_{2}^{0})$. 
Then, a cyclic chain of length $m$ is contained in a connected component $S'$ of the surface $S_{g}^{0} \setminus \mathrm{Int}N(\alpha)$, where $\alpha$ is a  closed curve corresponding to the graph $K_1$. 
The component $S'$ is homeomorphic to $S_{g'}^{p'}$, where $g' = g-1$ and $p' \leq 2$. 
Hence, we have $m \leq 2(g-1) + 2 + 1 = 2g+1$ by Lemma \ref{cyclic_full_subgraph}. 

{\bf Case of $g \geq 2, \ p \geq 1$}. 
Suppose that $C_{m}^c * K_1 \leq \mathcal{C}(S_{g}^{p})$. 
Then, a cyclic chain of length $m$ is contained in a connected component $S'$ of the surface $S_{g}^{p} \setminus \mathrm{Int}N(\alpha)$, where $\alpha$ is a  closed curve corresponding to the graph $K_1$. 
The component $S'$ is homeomorphic to $S_{g'}^{p'}$ with one of the following: 
\begin{enumerate}
 \item[(a)] $g' = g$ and $1 \leq p' \leq p-1$, 
 \item[(b)] $g' \leq g-1$ and $p' \leq p+2$.  
\end{enumerate} 
Hence, we have $m \leq 2(g-1)+(p+2) = 2g+p$ by Lemma \ref{cyclic_full_subgraph}. 
\end{proof} 

We denote the path graph on $m$ vertices by $P_m$ (the geometric realization is homeomorphic to the unit interval). 

\begin{lemma}
Suppose that $p \geq 5$. 
If $P_{m}^c * K_2 \leq \mathcal{C}(S_{0}^{p})$, then $m \leq p-3$. 
 \label{0p_path_plus}
\end{lemma}
\begin{proof}
Suppose that $P_{m}^c * K_2 \leq \mathcal{C}(S_{0}^{p})$. 
Then, there is a linear chain $( \alpha_1, \ldots, \alpha_{m} )$ and closed curves $\beta_1$, $\beta_2$ on $S_{0, p}$ such that each $\beta_i$ is isotopic to none of $\{ \alpha_1, \ldots, \alpha_{m}, \beta_{j} \}$ and is disjoint from $\{ \alpha_1, \ldots, \alpha_{m}, \beta_j \}$ for $j \neq i$. 
A connected component $S'$ of $S_{0, p} \setminus \mathrm{Int}N(\beta_1) \cup \mathrm{Int}N(\beta_2)$ contains the linear chain $( \alpha_1, \ldots, \alpha_{m} )$. 
Note that the connected component $S'$ is a sphere with at most $p-2$ holes. 
Hence, by \cite[Theorem 2.2]{Katayama--Kuno18}, the length of any linear chain on $S_{0, p} \setminus \mathrm{Int}N(\beta_1) \cup \mathrm{Int}N(\beta_2))$ is at most $p-3$. 
Therefore, $m \leq p-3$. 
\end{proof}

\begin{lemma}
Suppose that $p \geq 2$. 
If $P_{m}^c * K_2 \leq \mathcal{C}(S_{1}^{p})$, then $m \leq p$. 
 \label{1p_path_plus}
\end{lemma}
\begin{proof}
We may assume that $m \geq p+1 \geq 3$. 
Suppose that $P_{m}^c * K_2 \leq \mathcal{C}(S_{1}^{p})$. 
Then, there is a linear chain $( \alpha_1, \ldots, \alpha_{m} )$ and closed curves $\beta_1$, $\beta_2$ on $S_{1, p}$ such that $\beta_i$ is isotopic to none of $\{ \alpha_1, \ldots, \alpha_{m}, \beta_{j} \}$ and is disjoint from $\{ \alpha_1, \ldots, \alpha_{m}, \beta_{j} \}$ for $i=1, 2$ and $j \neq i$. 
A connected component of $S_{1}^{p} \setminus \mathrm{Int}N(\beta_1) \cup \mathrm{Int}N(\beta_2)$ contains the linear chain $( \alpha_1, \ldots, \alpha_{m} )$. 
From results in \cite[Theorem 2.2]{Katayama--Kuno18}, it follows that the length of any linear chain on $S_{0, p} \setminus \mathrm{Int}N(\beta_1) \cup \mathrm{Int}N(\beta_2))$ is at most $p$. 
Therefore, we have $m \leq p$. 
\end{proof}

\begin{lemma}
Suppose that $m \geq 1$, $g \geq 2$ and $p \geq 0$. 
If $P_{2g+p+1}^c * K_m \leq \mathcal{C}(S_{g}^{p})$, then $m \leq g-2$. 
 \label{2g_p_1}
\end{lemma}
\begin{proof}
We give a proof by induction on the genus $g$. 
Suppose that $P_{2g+p+1}^c * K_m \leq \mathcal{C}(S_{g}^{p})$. 
Then, there are mutually non-isotopic closed curves $\alpha_1, \ldots, \alpha_{2g+p+1}$, $\beta_1, \ldots, \beta_{m}$ satisfying the following conditions: (i) $\{ \alpha_1, \ldots, \alpha_{2g+p+1} \}$ is a linear chain of length $2g+p+1$ and (ii) $\beta_j$ is disjoint from the other curves $\alpha_1, \ldots, \alpha_{2g+p+1}$ and $\beta_1, \ldots, \beta_{j-1}, \beta_{j+1}, \ldots, \beta_{m}$. 

{\bf Case of $g=2, \ p \geq 0$}. 
Suppose, on the contrary, that $m \geq 1$. 
By conditions (i) and (ii), there is a connected component $S'$ of $S_{2}^{p} \setminus \mathrm{Int}N(\beta_1)$ containing a linear chain $( \alpha_1, \ldots, \alpha_{p+5})$. 
However, \cite[Theorem 2.2]{Katayama--Kuno18} implies that the length of any linear chain on $S'$ is at most $p+4$.  
Therefore, the initial assumption $m \geq 1$ must be false.  

{\bf Case of $g \geq 3, \ p \geq 0$}. 
Assuming $m \geq g-2 \geq 1$, we will prove that $m = g-2$. 
A connected component $S'$ of $S_{g}^{p} \setminus \mathrm{Int}N(\beta_1)$ contains a linear chain $(\alpha_1, \ldots, \alpha_{2g+p+1})$. 
By \cite[Theorem 2.2]{Katayama--Kuno18}, we may assume that $S'$ is homeomorphic to $S_{g-1}^{p+2}$. 
Hence, $P_{2g+p+1}^c * K_{m-1} \leq \mathcal{C}(S_{g-1}^{p+2})$. 
The induction hypothesis implies $m-1 = (g-1)-2$, as desired. 
\end{proof}

\begin{lemma}
Suppose that $m \geq 1$, $g \geq 2$ and $p \geq 0$. 
If $P_{2g+p}^c * K_m \leq \mathcal{C}(S_{g}^{p})$, then $m \leq g-1$. 
 \label{2g_p}
\end{lemma}
\begin{proof}
Suppose that $P_{2g+p} * K_m \leq \mathcal{C}(S_{g}^{p})$. 
Then, we have mutually non-isotopic closed curves $\alpha_1, \ldots, \alpha_{2g+p}, \beta_1, \ldots, \beta_{m}$ satisfying the following conditions: 
\begin{enumerate}
 \item[(i)] $(\alpha_1, \ldots, \alpha_{2g+p})$ is a linear chain and 
 \item[(ii)] $\beta_j$ is disjoint from the other curves $\alpha_1, \ldots, \alpha_{2g+p}$, $\beta_1, \ldots, \beta_{j-1}, \beta_{j+1}, \ldots, \beta_{m}$. 
\end{enumerate}
We give a proof by induction on the genus $g$. 

{\bf Case of $g=2, \ p \geq 0$}. 
Suppose, on the contrary, that $m \geq 2 \ (> g-1)$. 
From properties (i) and (ii), it follows that there is a connected component $S'$ of $S_{2, p} \setminus \mathrm{Int}N(\beta_1 \cup \beta_2)$ containing a linear chain $\{ \alpha_1, \ldots, \alpha_{p+4} \}$. 
We may assume that $S'$ is homeomorphic to either $S_{2}^{p-2}$ or $S_{1}^{p+1}$. 
However, neither of the surfaces contain a linear chain of length $p+4$. 
Therefore, the initial assumption $m \geq 2$ must be false, and so $m \leq 1$. 

{\bf Case of $g \geq 3, \ p \geq 0$}. 
Assuming $m \geq g-1 \geq 2$, we will prove that $m = g-1$. 
From properties (i) and (ii), it follows that there is a connected component $S'$ of $S_{g}^{p} \setminus \mathrm{Int}N(\beta_1)$ containing a linear chain $(\alpha_1, \ldots, \alpha_{2g+p})$. 
In the case where $S'$ is a surface of genus $g$ with $\leq p-1$ holes, Lemma \ref{2g_p_1} implies $m-1 \leq g-2$, i.e., $m \leq g-1$. 
Assume that $S'$ is a surface of genus $\leq g-1$ with $\leq p+2$ holes. 
Since $P_{2g+p}^c * K_{m-1} \leq \mathcal{C}(S')$, we have that $m-1 \leq (g-1)-1$ by the induction hypothesis. 
In every case, we have the desired inequality $m \leq g-1$. 
\end{proof}

\section{Right-angled Artin subgroups in mapping class groups \label{raags_in_mcgs}}

For a simple finite graph $\Gamma$, the {\it right-angled Artin group} $A(\Gamma)$ is the group with the following finite presentation: 
$$ A(\Gamma) := \langle V(\Gamma) \mid v_iv_j = v_jv_i \ \mbox{if} \ \{ v_i, v_j \} \in E(\Gamma) \rangle .$$
Here, $V(\Gamma)$ is the vertex set of $\Gamma$ and $E(\Gamma)$ is the edge set of $\Gamma$. 
The main purpose of this section is to prove Theorems \ref{cyc_mcg} and \ref{cyc_plus_z_mcg}. 
Deciding whether a right-angled Artin group can be embedded in the mapping class group of a surface is supposed to be difficult; at the moment the authors do not have an algorithm that solves the problem. 
Koberda in \cite{Koberda} proved that, for any family of mutually non-isotopic essential simple closed curves on a surface, sufficiently high powers of Dehn twists about the curves generate a right-angled Artin group in the mapping class group of the surface. 
This theorem is phrased as follows. 

\begin{theorem}[Koberda's embedding theorem]
Let $S$ be a surface of negative Euler characteristic and let $\Gamma$ be a finite graph. 
If $\Gamma \leq \mathcal{C}(S)$, then $A(\Gamma) \hookrightarrow \mathrm{Mod}(S)$. 
\end{theorem}

Kim--Koberda in \cite{Kim--Koberda} introduced combinatorial embeddings of right-angled Artin groups into the mapping class groups of surfaces and proved that such embeddings are normal (Lemma \ref{KK_normal_form} below). 
We say that a group homomorphism $\psi \colon A(\Lambda) \rightarrow \mathrm{Mod}(S)$ satisfies {\it condition (KK)} if $\psi(v)$ is a product of mutually commutative Dehn twists for every $v \in V(\Lambda)$. 
Note that, if $\psi$ satisfies condition (KK), the support curves of Dehn twists appearing in $\psi(v)$ induce a clique in $\mathcal{C}(S)$ if $\psi$ satisfies condition (KK). 

\begin{lemma} \label{KK_normal_form}
Let $S$ be a surface of negative Euler characteristic and without a boundary and let $\Gamma$ be a finite graph. 
If $A(\Gamma) \hookrightarrow \mathrm{Mod}(S)$, then there is an embedding $A(\Gamma) \hookrightarrow \mathrm{Mod}(S)$ that satisfies condition (KK). 
\end{lemma}

For an embedding $\psi \colon A(\Lambda) \hookrightarrow \mathrm{Mod}(S)$ satisfying condition (KK) and a vertex $u$ of $\Lambda$, $\mathrm{supp}(\psi(u))$ denotes the subset of $V(\mathcal{C}(S))$ corresponding to the supports of Dehn twists that appears in $\psi(u)$. 
To simplify the notation, we write $\mathrm{supp}(\psi)$ for $\cup_{u \in V(\Lambda)} \mathrm{supp}(\psi(u))$. 

\begin{theorem}
Let $m$ be an integer $\geq 5$, $S$ a surface of negative Euler characteristic and without a boundary, and $\Lambda$ a graph isomorphic to $C_m^c$. 
Suppose that there is an embedding $\psi \colon A(\Lambda) \hookrightarrow \mathrm{Mod}(S)$ with condition (KK). 
Then there is a full subgraph $\Delta \leq \mathcal{C}(S)$ that is induced by some vertices in $\mathrm{supp}(\psi)$ and has exactly one of the following properties: 
\begin{enumerate}
 \item[(1)] $\Delta \cong C_n^c$ (for some $n \geq m$), 
 \item[(2)] $\Delta \cong P_{n}^c$ (for some $n \geq m+1$), 
 \item[(3)] $\Delta \cong P_{m-1}^c * K_m$. 
 Furthermore, for each vertex $v$ of $\Delta$, there is a vertex in $\mathrm{supp}(\psi)$ that is not adjacent to $v$. 
\end{enumerate}
 \label{cyc_path_mcg_cor}
\end{theorem}
\begin{proof}
By Lemma \ref{KK_normal_form}, we may assume that the embedding $\psi \colon A(\Lambda) \hookrightarrow \mathrm{Mod}(S)$ satisfies condition (KK). 
We name the vertices $u_1, u_2, \ldots, u_m$ of $\Lambda \cong C_{m}^c$ to indicate that $u_i$ is not adjacent to $u_{i+1}$ ($1 \leq i \leq m-1$) and that $u_m$ is not adjacent to $u_1$. 
Let $\Gamma[\psi]$ denote the full subgraph induced by $\mathrm{supp}(\psi)$ in $\mathcal{C}(S)$. 
There are $m$ mutually distinct full embeddings $\iota_i \colon P_{m-1}^c \rightarrow C_m^c$, up to inversion. 
By \cite[Proposition 3.9]{Katayama--Kuno18}, each $\iota_i$ induces a full embedding $\tilde{\iota}_i \colon P_{m-1}^c \rightarrow \Gamma[\psi]$. 
We denote the image of $\tilde{\iota}_i$ by $J_i$. 
Since $J_i \cong P_{m-1}^c$, 
\begin{enumerate}
 \item[($*$)] for each vertex $v$ of $J_i$, there is a vertex of $J_i$ that is not adjacent to $v$ in $\Gamma[\psi]$. 
\end{enumerate}
Assuming that neither assertion (1) nor (2) holds, we shall prove that assertion (3) holds. 
First, consider the case where $J_1, \ldots, J_m$ induces a full subgraph isomorphic to $*^{m}P_{m-1}^c$ in $\Gamma[\psi]$. 
Since $m \geq 5$, $P_{m-1}^c$ contains a full subgraph isomorphic to $K_2$. 
This implies that $P_{m-1}^c * K_{2m-2} \leq P_{m-1} * (*^{m-1} P_{m-1}^c) = *^{m} P_{m-1}$. 
Hence, we have an isomorphic copy $\Delta$ of $P_{m-1}^c * K_{2m-2}$ in $\Gamma[\psi]$ ($m \geq 5$ implies $2m-2 \geq m$). 
Moreover, since each $J_i$ has the above property $(*)$, for each vertex $v$ of $\Delta$, there is a vertex of $*_{i=1}^{m} J_i \leq \Gamma[\psi]$ that is not adjacent to $v$ in $\Gamma[\psi]$. 
Therefore, assertion (3) holds. 

Now, let us consider the case where some pair of $J_1, \ldots, J_m$ does not induce a graph-join in $\Gamma[\psi]$. 
Let $\phi$ be the correspondence between $\Lambda$ and $\Gamma[\psi]$ such that $\phi \circ \tilde{\iota}_i = \iota_i$ given in \cite[Proposition 3.9]{Katayama--Kuno18}. 
Then, the restriction of $\phi$ to $\cup_{i=1}^m J_i$ is a (surjective) map onto $\Lambda$. 
To simplify the notation, let $\phi$ denote this restriction map $\phi |_{\cup_{i=1}^m J_i}$ (we will use the restriction rather than the original correspondence). 
Moreover, $\phi$ is a graph embedding of $J_i^c \cong P_{m-1}$ into $\Lambda$. 
Also, by changing indices, if necessary, we may assume that $J_1$ and $J_i$ ($i \neq 1$) does not induce a graph-join. 
In other words, there is a pair of vertices of $J_1$ and $J_i$ that are not adjacent in $\Gamma[\psi]$. 
To prove assertion (3), we have to prove the following claim. 

\begin{claim}
The full subgraph induced by $J_1$ and $J_i$ in $\Gamma[\psi]$ contains a full subgraph isomorphic to $P_{m}^c$. 
\label{longer_2}
\end{claim}
\begin{proof}[Proof of Claim \ref{longer_2}.]
Let $u_0$ be the element of a singleton set $V(\iota_{1}(P_{m-1})) \setminus V(\iota_{i}(P_{m-1}))$. 
Pick a vertex $v_0$ of $J_1$ so that $\phi (v_0) = u_0$. 
Since $\phi \circ \tilde{\iota}_i = \iota_i$, it follows that $u_0 \not\in \phi(V(J_i))$. 
Therefore, we have that $v_0 \not\in J_i$. 
In the following, we use both $\Gamma[\psi]$ and its complement graph, $\Gamma[\psi]^c$, for convenience.  
Pick a sequence $J:=(v_0, v_1, v_2, \ldots, v_s, v_{s+1}, \ldots, v_{s+t})$ of vertices in $\Gamma[\psi]$ satisfying the following conditions. 
\begin{enumerate}
 \item[(a)] $v_0, v_1, v_2, \ldots, v_s \in V(J_1)$ and $v_{s+1}, v_{s+1}, \ldots, v_{s+t} \in V(J_i)$, 
 \item[(b)] $v_{j}$ and $v_{j+1}$ are adjacent in $\Gamma[\psi]^c$ (for all $0 \leq j \leq s+t-1$),  
 \item[(c)] the vertices $v_0, v_1, v_2, \ldots, v_s, v_{s+1}, \ldots, v_{s+t}$ are mutually distinct, 
 \item[(d)] $\phi(\{ v_0, v_1, v_2, \ldots, v_s, v_{s+1}, \ldots, v_{s+t} \}) = V(\Lambda)$ and $\phi(v_0) \not\in \phi(J_i)$, 
 \item[(e)] $s=0$, or $J_i$ and the full subgraph induced by $\{ v_0, v_1, \ldots, v_{s-1} \}$ make a graph-join in $\Gamma[\psi]$, 
 \item[(f)] $J_1$ and the full subgraph induced by $\{ v_{s+2}, \ldots, v_{s+t} \}$ make a graph-join in $\Gamma[\psi]$. 
\end{enumerate}
Before proving the existence of a sequence of vertices satisfying conditions (a), (b), (c), (d), (e) and (f), we prove that such a sequence induces a full subgraph isomorphic to $P_m^c$ in $\Gamma[\psi]$. 
Assume that we have a sequence $J$ as above. 
Conditions (a) and (d) imply $r:=s+t+1 \geq m$. 
Hence, from conditions (b) and (c), we obtain a graph embedding $\tau$ of $P_{r}$ into $\Gamma[\psi]^c$ that maps the vertices of $P_r$ to $J$.  
Now, we prove that this embedding is full. 
Pick an edge $e$ of $\Gamma[\psi]^c$ (not $\Gamma[\psi]$) joining two vertices of $J$. 
Consider the case where $\partial e \subset \{v_0, v_1, \ldots, v_s \}$. 
Since $J_1^c$ is a full subgraph of $\Gamma[\psi]^c$, we have that $e \in E(J_1^c)$. 
Then, conditions (a) and (b) imply $\partial e = \{ v_j, v_{j+1} \}$ for some $0 \leq j \leq s$. 
Hence, $e$ is contained in $\tau(P_{r})$. 
Similarly, in the case where $\partial e \subset V(J_i^c)$, $e$ is contained in $\tau(P_r)$. 
The remaining case is that $\partial e = \{ v_x, v_y \}$ for some $0 \leq x \leq s$ and $s+1 \leq y \leq s+t$. 
Notice that $0 \leq x < s$ together with condition (e) implies $s=0$ and $x=0$, because $e \in \Gamma[\psi]^c$. 
Therefore we have $x=s$. 
Similarly, condition (f) implies $y=s+1$. 
Thus, $e$ is contained in $\tau(P_{r})$, and therefore $\tau$ is a full embedding. 
Consequently, $\tau$ induces a full embedding of $P_r^c$ into $\Gamma[\psi]$ whose image is $J$. 

Now, let us construct a sequence satisfying the desired conditions. 
Pick a vertex $v_{\star}$ of $J_1^c \cong P_{m-1}$ such that $v_{\star}$ is adjacent to some vertex of $J_i^c$ in $\Gamma[\psi]^c$, and that the shortest path from $v_0$ to $v_{\star}$ in $J_1^c$ has no vertex adjacent to a vertex of $J_i^c$ in $\Gamma[\psi]^c$ other than $v_0$ and $v_{\star}$ ($v_{\star}$ might be $v_0$). 
Let $v_1, \ldots, v_{s-1}$ be the vertices of $J_1^c$ such that $(v_0, v_1, \ldots, v_{s-1}, v_{\star})$ is the shortest path from $v_0$ to $v_{\star}$ in $J_1^c$. 
It follows that $s+1 \leq m-1$. 
We put $v_s:= v_{\star}$. 
Since $\phi(\{ v_0, v_1, \ldots, v_{s-1}, v_s \})$ induces a path-graph on $\leq m-1$ vertices in $\Lambda^c = C_m$ and since $\phi(v_0) \not\in \phi(V(J_i))$, at least one end point $v^{\mathrm{end}}$ of the path-graph $J_i^c$ satisfies $\phi(v^{\mathrm{end}}) \not\in \{ \phi(v_0), \phi(v_1), \ldots, \phi(v_{s-1}), \phi(v_{s}) \}$. 
We should note that $\phi(v^{\mathrm{end}})$ is adjacent to $\phi(v_0)$ in $\Lambda^c$. 
Let $v^{\star}$ be the vertex of $J_i^c$ which is adjacent to $v_{s}$ in $\Gamma[\psi]^c$ and is nearest to $v^{\mathrm{end}}$ in $J_i^c$. 
Put $v_{s+1}:= v^{\star}$ and pick the vertices $v_{s+2}, \ldots, v_{s+t-1}$ of $J_i^c$ in such a that $(v_{s+1}, v_{s+2}, \ldots, v_{s+t-1}, v^{\mathrm{end}})$ is the shortest path from $v^{\mathrm{end}}$ to $v_{s+1}$. 
Let us show that the sequence 
$$(v_0, v_1, \ldots, v_s, v_{s+1}, \ldots, v_{s+t-1}, v_{s+t}:= v^{\mathrm{end}})$$ 
satisfies conditions (a), (b), (c), (d), (e) and (f). 

Conditions (a) and (b): Obvious. 

Condition (c): Since $( v_0, v_1, \ldots, v_s )$ is the shortest path from $v_0$ to $v_s$ in $J_1^c$, any two vertices in $\{v_0, v_1, \ldots, v_s \}$ are distinct. 
Similarly, any two vertices in $\{ v_{s+1}, \ldots, v_{s+t} \}$ are distinct. 
Moreover, since the only edge joining $\{ v_1, \ldots, v_s \}$ to $\{ v_{s+1}, \ldots, v_{s+t} \}$ in $\Gamma[\psi]^c$ is $[v_s, v_{s+1}]$, we have that $\{v_1, \ldots, v_s \} \cap \{ v_{s+1}, \ldots, v_{s+t} \} = \emptyset$. 
Furthermore, $v_0$ is not contained in $\{ v_{s+1}, \ldots, v_{s+t} \}$, because $\phi(v_0)$ is not contained in $\phi(\{ v_{s+1}, \ldots, v_{s+t} \}) \subset \phi(V(J_i))$. 
Therefore, the vertices in 
$$\{v_0, v_1, \ldots, v_s, v_{s+1}, \ldots, v_{s+t-1}, v_{s+t} \}$$ 
are mutually distinct. 

Condition (d): By definition, we have $\phi(v_0) = u_0 \not\in \phi(J_i)$. 
So we have to prove that $\phi(\{ v_0, v_1, v_2, \ldots, v_s, v_{s+1}, \ldots, v_{s+t} \}) = V(\Lambda)$. 
Since $\phi(v_0) = u_0$, it is enough to show that $\phi(\{ v_1, v_2, \ldots, v_{s+t} \}) = V(\iota_{i}(P_{m-1}^c)) \ ( = V(\Lambda) \setminus \{ u_0 \} )$. 
Since $v_{s+t}$ is an end point of $J_i$ and since $v_{0} \in V(\Lambda^c \setminus V(\iota_i(P_{m-1}^c))$, $\phi(v_{s+t})$ is not adjacent to $\phi(v_0)$ in $\Lambda$. 
In addition, $\phi(v_1)$ is not adjacent to $\phi(v_0)$, because the restriction of $\phi$ to $J_1^c$ is a graph homomorphism and that $v_1$ is not adjacent to $v_0$. 
Therefore, $\phi(v_1)$ and $\phi(v_{s+t})$ are the end points of the path graph $\iota_i(P_{m-1})$. 
Hence, the path $(\phi(v_1), \phi(v_2) \ldots, \phi(v_{s+t}))$ joins the end points of $\iota_i(P_{m-1})$ in $\Lambda^c \cong C_m$. 
Since the restriction of $\phi$ to $J_1^c$ is a graph embedding and since $\phi(J_i) \ (= \iota(P_{m-1}))$ does not contain $u_0 = \phi(v_0)$, it holds that the path $(\phi(v_1), \phi(v_2) \ldots, \phi(v_{s+t}))$ in $\Lambda^c$ does not pass $\phi(v_0)$. 
Thus, $(\phi(v_1), \phi(v_2) \ldots, \phi(v_{s+t}))$ covers $V(\iota(P_{m-1}^c))$, as desired. 

Conditions (e) and (f): It is enough to show that $v_{s+t}$ is adjacent to $v_0$ in $\Gamma[\psi]$. 
To this end, we shall show that $(v_0, v_1, \ldots, v_s, v_{s+1}, \ldots, v_{s+t})$ induces a full subgraph isomorphic to $C_{s+t+1}^c$ in $\Gamma[\psi]$ (assertion (1)) by assuming that $v_{s+t}$ is not adjacent to $v_0$ in $\Gamma[\psi]$ (hence, they are adjacent in $\Gamma[\psi]^c$). 
This assumption together with the fact that $(v_0, v_1, \ldots, v_s, v_{s+1}, \ldots, v_{s+t})$ satisfies conditions (b) and (c) implies that there is a natural graph embedding $\alpha \colon C_{s+t+1} \rightarrow \Gamma[\psi]^c$. 
Moreover, it follows that $\alpha$ is full from the choice of paths $(v_0, v_1, \ldots, v_{s})$ and $(v_{s+t}, v_{s+t-1}, \ldots, v_{s+1})$. 
Hence, $(v_0, v_1, \ldots, v_s, v_{s+1}, \ldots, v_{s+t})$ induces a full subgraph isomorphic to $C_{s+t+1}$ in $\Gamma[\psi]^c$. 
However, by assumption, $\Gamma[\psi]$ has no full subgraph isomorphic to $C_{s+t+1}^c$, and hence $v_{s+t}$ is adjacent to $v_0$ in $\Gamma[\psi]$. 
Thus, the sequence $(v_0, v_1, \ldots, v_s, v_{s+1}, \ldots, v_{s+t})$ satisfies conditions (e) and (f). 

As shown above, the sequence $(v_0, v_1, \ldots, v_s, v_{s+1}, \ldots, v_{s+t})$ with the above conditions induces a full subgraph isomorphic to $P_{s+t+1}^c$ ($s+t+1 \geq m$) in $\Gamma[\psi]$. 
\end{proof}

We also need to prove the following claim. 

\begin{claim}
If $J_1^c \cup J_i^c \cup J_j^c$ ($j \neq 1$ and $j \neq i$) induces a connected full subgraph in $\Gamma[\psi]^c$, then the full subgraph of $\Gamma[\psi]$ induced by $J_1 \cup J_i \cup J_j$ contains a full subgraph isomorphic to $P_{m+1}^c$. 
\label{longer_3}
\end{claim}
\begin{proof}[Proof of Claim \ref{longer_3}.]
Suppose that $J_1^c \cup J_i^c \cup J_j^c$ induces a connected full subgraph in $\Gamma[\psi]^c$. 
By changing indices, if necessary, we may assume that $J_1^c$ and $J_i^c$ are connected by an edge in $\Gamma[\psi]^c$, and that $J_1^c$ and $J_j^c$ are connected by an edge in $\Gamma[\psi]^c$. 
Let $(v_0, v_1, \ldots, v_s, v_{s+1}, \ldots, v_{s+t})$ be the sequence of vertices in $V(J_1) \cup V(J_i)$ constructed in the proof of Claim \ref{longer_2}. 
If $s+t+1 \geq m+1$; then, the claim holds. 
Thus, we may assume that $s+t+1 = m$. 
Then, condition (c) implies that the restriction of $\phi$ to the sequence $(v_0, v_1, \ldots, v_s, v_{s+1}, \ldots, v_{s+t})$ is a graph embedding of $P_m^c$ into $C_m^c$. 
Moreover, as in the proof of Claim \ref{longer_2}, $\phi(v_{s+t})$ is adjacent to $\phi(v_0)$ in $\Gamma[\psi]^c$. 

(i) Case where $J_1$ has a vertex $v_a$ such that $\phi(v_a) = \phi(v_{s+t})$. 
A key observation is that $\phi(E(J_1^c) \cup E(J_{i}^c)) = E(\Lambda^c) = E(C_m)$.  
We claim that $(v_a, v_0, v_1, \ldots, v_s, v_{s+1}, \ldots, v_{s+t})$ induces a full subgraph isomorphic to $P_{s+t+2}^c$ ($s+t+2 \geq m+1$) in $\Gamma[\psi]$. 
To see this, we will show the following: 
\begin{enumerate}
 \item[$\bullet$] $v_a$ is adjacent to $v_0$ in $\Gamma[\psi]^c$, 
 \item[$\bullet$] $v_a \not \in \{v_0, \ldots, v_{s+t} \}$, and 
 \item[$\bullet$] each vertex in $\{ v_1, \ldots, v_{s+t} \}$ is adjacent to $v_a$ in $\Gamma[\psi]$. 
\end{enumerate}
These facts together with the fact that the sequence $(v_0, \ldots, v_{s+t})$ induces a full subgraph isomorphic to $P_{m}^c$ imply that the longer sequence $(v_a, v_0, \ldots, v_{s+t})$ induces a full subgraph isomorphic to $P_{m+1}^c$ in $\Gamma[\psi]$. 
First, we show that $v_a$ is not adjacent to $v_0$ in $\Gamma[\psi]$. 
Since $\phi(v_a) = \phi(v_{s+t})$ and since $\phi(v_{s+t})$ is adjacent to $\phi(v_0)$ in $\Lambda^c$, $\phi(v_a)$ is adjacent to $\phi(v_0)$ in $\Lambda^c$. 
The restriction of $\phi$ to $J_1^c$ is a full embedding of $J_1^c$ into $\Lambda^c$, because $J_1^c \cong P_{m-1}$ and $\phi \circ \tilde{\iota_1} = \iota_1$. 
Therefore, $v_a$ must be adjacent to $v_0$. 
Next, we show that $v_a \not \in \{v_0, \ldots, v_{s+t} \}$. 
The path $(\phi(v_0), \ldots, \phi(v_{s}) )$ ($s+1 \leq m-1$) in $\Lambda^c$ cannot reach $\phi(v_{s+t})$, because $\phi(v_0)$ is adjacent to $\phi(v_{s+t})$ in $\Lambda^c$ and $\phi(v_1) \neq \phi(v_{s+t})$. 
Hence, we have $v_a \not\in \{ v_0, v_1, \ldots, v_{s} \}$. 
In addition, no vertex in $\{ v_{s+1}, v_{s+2}, \ldots, v_{s+t-1} \}$ is mapped to $\phi(v_{s+t})$. 
Therefore, $v_{a} \not\in \{ v_{s+1}, v_{s+2}, \ldots, v_{s+t-1}  \}$. 
By assumption, $\Gamma[\psi]$ has no full subgraph isomorphic to $C_{m}^c$, and hence $v_{a} \neq v_{s+t}$ (if not, the path $(v_0, v_1, \ldots, v_{s+t}=v_a)$ induces $C_{s+t+1} = C_m$ in $\Gamma[\psi]^c$). 
As a result, we have $v_{a} \not\in \{ v_0, v_1, \ldots, v_s, v_{s+1}, \ldots, v_{s+t} \}$. 
Next, we show that each vertex in $\{ v_1, \ldots, v_s, v_{s+1}, \ldots, v_{s+t} \}$ is adjacent to $v_a$ in $\Gamma[\psi]$. 
Since $(\phi(v_0), \ldots, \phi(v_{s+t}))$ induces an embedding of $P_{m}$ into $\Lambda^c \cong C_m$, no vertex in $(\phi(v_1), \ldots, \phi(v_{s+t-2}))$ is adjacent to $\phi(v_a)$ in $\Lambda^c$. 
In other words, each vertex in $(\phi(v_1), \ldots, \phi(v_{s+t-2}))$ is adjacent to $\phi(v_a)$ in $\Lambda$. 
Since $\phi$ is a graph homomorphism from the full subgraph induced by $J_1^c \cup J_i^c$ in $\Gamma[\psi]^c$ to $\Lambda^c$, no vertex in $(v_1, \ldots, v_{s+t-2})$ is adjacent to $v_a$ in $\Gamma[\psi]^c$. 
Therefore, each vertex in $(v_1, \ldots, v_{s+t-2})$ is adjacent to $v_a$ in $\Gamma[\psi]$. 
Note that $(v_a, v_0, v_1, \ldots, v_{s+t-2})$ induces a full subgraph isomorphic to $P_{m-1}^c$ in $\Gamma[\psi]$. 
This implies that $v_a$ is adjacent to $v_{s+t-1}$ in $\Gamma[\psi]$, because there is no full subgraph isomorphic to $C_m^c$ in $\Gamma[\psi]$. 
Also, $v_a$ is adjacent to $v_{s+t}$, because $\psi$ satisfies condition (KK) and $\phi(v_a) = \phi(v_{s+t})$. 
Thus each vertex in $\{ v_1, \ldots, v_{s+t} \}$ is adjacent to $v_a$ in $\Gamma[\psi]$. 
The above argument shows that $(v_a, v_0, \ldots, v_{s+t})$ induces a full subgraph isomorphic to $P_{m+1}^c$ in $\Gamma[\psi]$

(ii) Case where $J_1$ has no vertex mapped to $\phi(v_{s+t})$. 
In this case, it holds that $\phi(v_{s+t}) \not\in \phi(J_1^c)$ and $\phi(v_0) \not\in \phi(J_i^c)$. 
Hence, $\phi(E(J_1^c) \cup E(J_{i}^c)) \subsetneq E(\Lambda^c) = E(C_m)$. 
Note that $J_j^c$ has an edge $[v_b, v_c]$ such that $\phi(v_b) = \phi(v_{s+t})$ and $\phi(v_c) = \phi(v_0)$. 
Suppose that $J_{j}^c$ contains a vertex mapped to $\phi(v_{s+t-1})$. 
Accordingly, we have $\phi(E(J_1^c)) \cup \phi(E(J_j^c)) = E(\Lambda^c)$. 
By applying Claim \ref{longer_2} and the argument in (i) to $J_1 \cup J_j$ instead of $J_1 \cup J_i$, we have that the full subgraph induced by $J_1 \cup J_j$ contains $P_{m+1}^c$. 
Hence, we may assume that $\phi(J_j^c)$ does not contain $\phi(v_{s+t-1})$. 
Then we have that $v_0$ is an endpoint of $J_1^c$, $v_b$ is an endpoint of $J_j^c$, and that $\phi(E(J_i^c)) \cup \phi(E(J_{j}^c)) = E(\Lambda^c)$. 
So again by Claim \ref{longer_2} and (i), we may assume that there is no edge between $J_i^c$ and $J_j^c$ in $\Gamma[\psi]^c$. 
We can now construct a path $\gamma$ from $v_{s+t} \in J_i^c$ to an endpoint of $J_j^c$ which passes $J_1^c$ and $[v_b, v_c]$ in this order as follows. 

Case (ii)-(A). 
Let $v_e$ be an an endpoint $v_{e}$ of $J_1^c \cong P_{m-1}$ other than $v_0$. 
If $v_b$ is not adjacent to $v_{e}$ in $\Gamma[\psi]^c$, then $\gamma$ can be obtained by connecting $(v_{s+t}, \ldots, v_{s+1}, v_{s}, \ldots, v_{d})$ with a subpath $S$ of $J_j^c$ starting from an internal vertex of $J_j^c$ and ending at $v_b$, where $v_d \in \{v_0, v_1, \ldots, v_s \}$ and no vertex between $v_s$ and $v_d$ is adjacent to a vertex of $J_j^c$ in $\Gamma[\psi]^c$. 

Case (ii)-(B). 
If $v_b$ is adjacent to $v_{e}$ in $\Gamma[\psi]^c$, then $\gamma$ can be obtained by connecting the shortest path from $v_{s+t}$ to $v_e$ in the full subgraph induced by $J_1^c \cup J_i^c$ with the entire $J_j^c$. 

Let us show that $\gamma$ induces a full subgraph $P$ isomorphic to a path graph of length at least $m+1$ in $\Gamma[\psi]^c$. 
By our construction, $\phi(\gamma)$ passes $\phi(v_{s+t})$ twice and all of the vertices in $\Lambda^c$, so the number of vertices contained in $\gamma$ is at least $m+1$. 
It is easy to see that $\gamma$ does not have an overlap. 
Moreover, we can show that $\gamma$ has no shortcut as follows. 

Case (ii)-(A). Both $(v_{s+t}, \ldots, v_{s+1}, v_{s}, \ldots, v_{c})$ and $S$ have no shortcut. 
Since there is no edge between $J_i^c$ and $J_j^c$, no vertex of $S$ is adjacent to $(v_{s+t}, \ldots, v_{s+1})$ in $\Gamma[\psi]^c$. 
By the properties of $v_d$, there is no edge between $S$ and $(v_s, \ldots, v_d)$. 

Case (ii)-(B). Suppose, on the contrary, that there is an edge $e$ connecting a vertex of $J_1^c$ and a vertex of $J_j^c$ other than $v_b$. 
Then $J_1^c \cup J_j^c$ contains a full subgraph isomorphic to $C_n$ for some $n \geq m$, because $[v_e, v_b]$ is an edge in $\Gamma[\psi]^c$. 
This does not happen, so there is exactly one edge, $[v_e, v_b]$, between $J_1^c$ and $J_j^c$. 
This implies that there is no shortcut in $\gamma$, because there is no edge between $J_i^c$ and $J_j^c$. 

By the above argument, $\gamma$ induces a full subgraph isomorphic to $P_{n}$ for some $n \geq m+1$, as desired. 
\end{proof}

We denote by $F$ the full subgraph of $\Gamma[\psi]^c$ induced by $J_1^c, \ldots, J_m^c$. 
Since two of $J_1^c, \ldots, J_m^c$, say $J_1^c$ and $J_i^c$, are connected by an edge in $\Gamma[\psi]^c$, Claim \ref{longer_2} implies that $F$ contains a full subgraph $P_{m} \sqcup \check{F}$, where $\check{F}$ is a full subgraph of $\Gamma[\psi]^c$ induced by $J_2^c, \ldots, J_{i-1}^c, J_{i+1}^c, \ldots, J_m^c$. 
Claim \ref{longer_3} together with the assumption $P_{m+1}^c \not\leq \Gamma[\psi]$ ensures that no triple of $J_1^c, \ldots, J_m^c$ is connected in $\Gamma[\psi]^c$. 
Hence, there is no edge between the isomorphic copy of $P_m$ and $\check{F}$ in the above union. 
Let $x$ denote the number of connected components of $F$ induced by pairs of $J_1^c, \ldots, J_m^c$, and let $y$ denote the number of other connected components of $F$. 
Then, we have $2x + y = m$. 
In order to prove that there is a full subgraph isomorphic to $P_{m} \sqcup (\sqcup^{m} K_1)$ in $F$, we will find $\sqcup^{m} K_1$ in $\check{F}$. 
Note that $\check{F}$ has $x-1$ connected components that contain full subgraphs isomorphic to $P_m$ and $y$ connected components that do not contain $P_m$ (but contain $P_{m-1}$). 
Since $m \geq 5$, $P_m$ contains a full subgraph isomorphic to $\sqcup^{3} K_1$. 
On the other hand, $P_{m-1}$ contains a full subgraph isomorphic to $\sqcup^{2} K_1$. 
Therefore, $\check{F}$ contains a full subgraph isomorphic to $\sqcup^{3(x-1) + 2y} K_1$. 
The equality $2x + y = m$ implies $3(x-1) + 2y = \frac{3m + y - 6}{2}$. 
We now claim that $\frac{3m + y - 6}{2} \geq m$. 
To see this, assume that $m=5$. 
Then, at least one component should be induced by a single $J_k$, and hence $y \geq 1$. 
Therefore, we have $\frac{3m + y - 6}{2} \geq m$. 
Next, we assume that $m \geq 6$. 
Then we have $3m-6 \geq 2m$, and hence, $\frac{3m + y - 6}{2} \geq m$. 
Thus, $\frac{3m + y - 6}{2} \geq m$ for all $m \geq 5$. 
Now it follows that $P_{m} \sqcup (\sqcup^{m} K_1) \leq P_{m} \sqcup (\sqcup^{\frac{3m + y - 6}{2}} K_1) \leq P_{m} \sqcup \check{F} \leq F \leq \Gamma[\psi]^c$. 
Namely, we have $P_m^c * K_m \leq \Gamma[\psi]$. 
Property $(*)$ implies that for each vertex $v$ of the full subgraph $P_{m}^c * K_m \leq \Gamma[\psi]$, there is a vertex of $\Gamma[\psi]$, which is not adjacent to $v$ in $\Gamma[\psi]$, and therefore assertion (3) holds. 
\end{proof}

The above theorem immediately implies the following. 

\begin{corollary}
Let $m$ be an integer $\geq 5$, $S$ a surface with negative Euler characteristic and without boundary, and $\Gamma$ a finite graph with $P_{m+1}^c \not \leq \mathcal{C}(S)$. 
Suppose that there is an embedding $\psi \colon A(C_m^c) \times A(K_1) \hookrightarrow \mathrm{Mod}(S)$. 
Then at least one of the following holds: 
\begin{enumerate}
 \item[(1)] $C_n * K_1 \leq \mathcal{C}(S)$ (for some $n \geq m$), 
 \item[(2)] $P_{m-1}^c * K_{m+1} \leq \mathcal{C}(S)$. 
\end{enumerate}
 \label{path_plus_mcg_cor}
\end{corollary}
\begin{proof}
We may assume that $\psi$ satisfies condition (KK). 
Restricting $\psi$ to $A(C_m^c)$ and $A(K_1)$, we have two embeddings $\psi_1 := \psi |_{A(C_m^c)} \colon$ $A(C_m^c) \hookrightarrow \mathrm{Mod}(S)$ and $\psi_2 := \psi |_{A(K_1)} \colon A(K_1) \hookrightarrow \mathrm{Mod}(S)$ with condition (KK). 
Applying Theorem \ref{cyc_path_mcg_cor} to $\psi_1$, we have that there is a full subgraph $\Delta$ of $\mathcal{C}(S)$ that is induced by some vertices in $\mathrm{supp}(\psi_1(V(C_m^c)))$ and is isomorphic to one of the following: 
\begin{enumerate}
 \item[$(1')$] $C_n$ (for some $n \geq m$) or 
 \item[$(2')$] $P_{m-1} * K_m$. 
 Furthermore, for each vertex $v$ of $P_{m-1}*K_m$, there is a vertex in $\mathrm{supp}(\psi_1(V(C_m^c)))$ that is not adjacent to $v$. 
\end{enumerate}
We only treat the case where $\Delta$ is isomorphic to $P_{m-1} * K_m$, because another case can be treated similarly. 
Suppose that $\Delta$ is isomorphic to $P_{m-1} * K_m$. 
Pick a vertex $u$ of $\mathrm{supp}(\psi_2(V(K_1)))$. 
Then every element of $\mathrm{supp}(\psi_1(V(C_m^c)))$ is adjacent to $u$ in $\mathcal{C}(S)$, because the images of $\psi_1$ and $\psi_2$ are commutative. 
So we have $u \not \in V(\Delta)$, and hence, $P_{m-1}^c * K_{m+1} \cong P_{m-1}^c * K_m * \{u \} \leq \mathcal{C}(S)$. 
Thus, assertion (2) holds. 
\end{proof}

Using the above results, we can decide whether $A(C_m^c)$ is embedded in the mapping class group of a surface.  

\begin{theorem}
$A(C_m^c)$ is embedded in $\mathrm{Mod}(S_{g, p})$ if and only if $m$ satisfies 
\begin{eqnarray*}
m \leq \left\{ \begin{array}{ll}
0 & ((g, p) \in \{ (0, 0), (0, 1), (0, 2), (0, 3)\}) \\
3 & ((g, p) \in \{ (0, 4), (1, 0), (1, 1) \}) \\
5 & ((g, p)=(1, 2), (0, 5)) \\
2g + 2 & (g \geq 2, \ p = 0) \\
2g + p + 1 & (g \geq 2, \ 1 \leq p \leq 2) \\
2g + p & (\mbox{otherwise}). \\
\end{array} \right.
\end{eqnarray*}
 \label{cyc_mcg}
\end{theorem}
\begin{proof}
{\bf Case of $(g, p) \in \{ (0, 0), (0, 1), (0, 2), (0, 3)\}$}. 
In this case, the mapping class group $\mathrm{Mod}(S_{g, p})$ is a finite group, and so the right-angled Artin group embedded in $\mathrm{Mod}(S_{g, p})$ should be trivial. 

{\bf Case of $(g, p)=(1, 1)$}. 
Note that $A(C_3^c) \cong F_3$ is embedded in $\mathrm{Mod}(S_{1, 1}) \cong \mathrm{SL}(2, \mathbb{Z})$. 
However, if $m \geq 4$, then $A(C_m^c)$ contains $\mathbb{Z}^2$. 
On the other hand, $\mathrm{SL}(2, \mathbb{Z})$ does not contain $\mathbb{Z}^2$. 
Hence, the assertion holds. 

{\bf Case of $(g, p) \in \{ (0, 5),  (1, 2) \}$}. 
The assertion in the case where $(g, p) = (0, 5)$, was proved in \cite[Proposition 4.1]{Katayama--Kuno18}. 
Suppose that $(g, p) = (1, 2)$. 
Koberda's embedding theorem together with Lemma \ref{cyclic_full_subgraph} shows that $A(C_5^c)$ is embedded in $\mathrm{Mod}(S_{1, 2})$. 
Next, we prove that $A(C_m^c) \hookrightarrow \mathrm{Mod}(S_{1, 2})$ implies $m \leq 5$. 
Note that $A(C_6^c)$ contains a subgroup isomorphic to $\mathbb{Z}^3$. 
On the other hand, the topological complexity of $S_{1, 2}$ is two. 
Hence, $A(C_6^c)$ is cannot be embedded in $\mathrm{Mod}(S_{1, 2})$. 

{\bf Case of $g \geq 2, \ p = 0$}. 
This is shown in \cite[Proposition 4.1]{Katayama--Kuno18}. 

{\bf Case of $g \geq 2, \ 1 \leq p \leq 2$}. 
Koberda's embedding theorem together with Lemma \ref{cyclic_full_subgraph} shows that $A(C_{2g+p+1}^c)$ is embedded in $\mathrm{Mod}(S_{g, p})$. 
Suppose that $A(C_m^c) \hookrightarrow \mathrm{Mod}(S_{g, p})$. 
Assuming $m \geq 2g+p+1$, we shall prove $m = 2g+p+1$. 
By \cite{Katayama--Kuno18}, $P_{m+1}^c$ is not a full subgraph of $\mathcal{C}(S_{g}^{p})$. 
Hence, from Theorem \ref{cyc_path_mcg_cor}, we have either $C_{m}^c \leq \mathcal{C}(S_{g, p})$ or $P_{m-1}^c * K_m \leq \mathcal{C}(S_{g}^{p})$. 
Lemma \ref{cyclic_full_subgraph} together with Lemma \ref{2g_p} implies that $m = 2g+p+1$. 

{\bf Case of $g=0, \ p \geq 6$}. 
See \cite[Proposition 4.1]{Katayama--Kuno18}. 

In the rest of the cases, we can give proofs similar to the case where $g \geq 2, \ 1 \leq p \leq 2$. 

{\bf Case of $g=1, \ p \geq 3$}. 
Koberda's embedding theorem together with Lemma \ref{cyclic_full_subgraph} shows that $A(C_{p+2}^c)$ is embedded in $\mathrm{Mod}(S_{1, p})$. 
Next, we suppose that $A(C_m^c) \hookrightarrow \mathrm{Mod}(S_{1, p})$. 
Furthermore, assume that $m \geq p+2 \geq 5$. 
Since $P_{p+3}^c \not \leq \mathcal{C}(S_{1}^{p})$ by \cite[Theorem 2.2]{Katayama--Kuno18}, either $C_{m}^c \leq \mathcal{C}(S_{1, p})$ or $P_{m-1}^c * K_m \leq \mathcal{C}(S_{1, p})$ holds. 
Hence, $m = p+2$. 

{\bf Case of $g \geq 2, \ p \geq 3$}. 
Koberda's embedding theorem together with Lemma \ref{cyclic_full_subgraph} shows that $A(C_{2g+p}^c)$ is embedded in $\mathrm{Mod}(S_{g, p})$. 
Suppose that $A(C_m^c) \hookrightarrow \mathrm{Mod}(S_{g, p})$. 
To the contrary, assume that $m \geq 2g+p+1$. 
Then, one of the following holds: $C_{m}^c \leq \mathcal{C}(S_{g, p})$ or $P_{2g+p+1}^c \leq \mathcal{C}(S)$ or $P_{2g+p}^c * K_m \leq \mathcal{C}(S_{g}^{p})$. 
However, this is impossible by \cite{Katayama--Kuno18} and Lemmas \ref{cyclic_full_subgraph} and \ref{2g_p}. 
Thus, $m \leq 2g+p$. 
\end{proof}

Similarly, we can decide whether $A(C_m^c) \times \mathbb{Z}$ is embedded in the mapping class group of a surface. 

\begin{theorem}
$A(C_m^c) \times \mathbb{Z}$ is embedded in $\mathrm{Mod}(S_{g, p})$ if and only if $m$ satisfies 
\begin{eqnarray*}
m \leq \left\{ \begin{array}{ll}
0 & (g, p) \in \{ (0, 4), (1, 1) \} \\
3 & (g, p) \in \{ (0, 5), (1, 2) \} \\
p-1 & (g=0, \ p \geq 6) \\
p + 2 & (g=1, \ p \geq 3) \\
2g + 1 & (g \geq 2, \ p=0)  \\
2g + p & (g \geq 2, \ p \geq 1). \\
\end{array} \right.
\end{eqnarray*}
 \label{cyc_plus_z_mcg}
\end{theorem}
\begin{proof}
Koberda's embedding theorem together with Lemma \ref{cyclic_plus_full_subgraph} and \cite[Proposition 4.1]{Katayama--Kuno18} implies the ``if part" of the theorem. 
Thus, we will prove the ``only if part". 

{\bf Case of $(g, p) \in \{ (0, 4), (1, 1) \}$}. 
If $m \geq 1$, then $A(C_m^c) \times \mathbb{Z}$ contains $\mathbb{Z}^2$. 
On the other hand, $\mathrm{Mod}(S_{0, 4})$ and $\mathrm{Mod}(S_{1, 1})$ do not contain $\mathbb{Z}^2$. 
Hence, $A(C_m^c) \times \mathbb{Z}$ is embedded in neither $\mathrm{Mod}(S_{0, 4})$ nor $\mathrm{Mod}(S_{1, 1})$ if $m \geq 1$. 

{\bf Case of $(g, p) \in \{ (0, 5), (1, 2) \}$}. 
Since the topological complexity of $S_{g, p}$ is two, the right-angled Artin group $A(C_4^c) \times \mathbb{Z}$ is not embedded in $\mathrm{Mod}(S_{g, p})$. 

{\bf Case of $g=0, \ p \geq 6$}. 
Suppose that $A(C_{m}^c) \times \mathbb{Z}$ is embedded in $\mathrm{Mod}(S_{0, p})$. 
We may assume that $m \geq p-1 \geq 5$. 
Then, there is an embedding $A(C_m^c) \times \mathbb{Z} \hookrightarrow \mathrm{Mod}(S_{0, p})$ satisfying condition (KK). 
As shown in \cite{Katayama--Kuno18}, we have $P_{m+1}^c \not\leq \mathcal{C}(S_{0}^{p})$. 
Then, from Corollary \ref{path_plus_mcg_cor}, either $C_m^c * K_1 \leq \mathcal{C}(S_{0, p})$ or $P_{p-2}^c * K_{m+1} \leq \mathcal{C}(S_{0}^{p})$ holds. 
Hence, by Lemmas \ref{cyclic_plus_full_subgraph} and \ref{0p_path_plus}, we have $m \leq p-1$. 

The remaining cases can be treated similarly. 

{\bf Case of $g=1, \ p \geq 3$}. 
Suppose that $A(C_{m}^c) \times \mathbb{Z}$ is embedded in $\mathrm{Mod}(S_{1, p})$. 
We may assume that $m \geq p+2 \geq 5$. 
Then there is an embedding $A(C_m^c) \times \mathbb{Z} \hookrightarrow \mathrm{Mod}(S_{0, p})$ satisfying condition (KK). 
By \cite{Katayama--Kuno18}, we have $P_{m+1}^c \not\leq \mathcal{C}(S_{g}^{p})$. 
From Corollary \ref{path_plus_mcg_cor}, it follows that either $C_m^c * K_1 \leq \mathcal{C}(S_{1}^{p})$ or $P_{p+1}^c * K_{m+1} \leq \mathcal{C}(S_{1}^{p})$. 
By Lemmas \ref{cyclic_plus_full_subgraph} and \ref{1p_path_plus}, we have $m \leq p+2$. 

{\bf Case of $g \geq 2, \ p=0$}. 
From Lemma \ref{2g_p_1}, it routinely follows that $P_{2g} * K_{2g+1} \not\leq \mathcal{C}(S_{g})$. 
In addition, $C_m^c * K_1 \leq \mathcal{C}(S_{g})$ only if $m \leq 2g+1$. 

{\bf Case of $g \geq 2, \ p \geq 1$}. 
Suppose that $A(C_{m}^c) \times \mathbb{Z}$ is embedded in $\mathrm{Mod}(S_{g, p})$. 
Assume to the contrary that $m \geq 2g+p+1$. 
There is an embedding $A(C_m^c) \times \mathbb{Z} \hookrightarrow \mathrm{Mod}(S_{g, p})$ satisfying condition (KK). 
Since $P_{2g+p+2}^c \not\leq \mathcal{C}(S_{g}^{p})$ by \cite[Theorem 2.2]{Katayama--Kuno18}, either $C_m^c * K_1 \leq \mathcal{C}(S_{g}^{p})$ or $P_{2g+p}^c * K_{2g+p+2} \leq \mathcal{C}(S_{g}^{p})$ holds. 
From Lemma \ref{2g_p_1}, it follows that $P_{2g+p}^c * K_{2g+p+1} \not\leq \mathcal{C}(S_{g}^{p})$. 
In addition, $C_m^c * K_1 \not\leq \mathcal{C}(S_{g}^{p})$. 
This is a contradiction. 
Thus, we have $m \leq 2g+p$. 
\end{proof}

\begin{remark}
Apart from a few exceptions, it holds that $A(C_m^c) \hookrightarrow \mathrm{Mod}(S_{g, p})$ if and only if $C_m^c \leq \mathcal{C}(S_{g}^{p})$. 
Similarly, $A(C_m^c) \times \mathbb{Z} \hookrightarrow \mathrm{Mod}(S_{g, p})$ if and only if $C_m^c * K_1 \leq \mathcal{C}(S_{g}^{p})$. 
\label{iff_raag}
\end{remark}

We end this section by noting a result for braid groups. 

\begin{theorem}[\cite{Katayama--Kuno18}]
\label{cyc_plus_z_braid}
Suppose $n \geq 3$. 
$A(C_{n+1}^c) \times \mathbb{Z} \hookrightarrow B_{n}$. 
\end{theorem}

\section{Proofs of Theorems \ref{main_1} and \ref{main_2}}

In this section, we prove Theorems \ref{main_1} and \ref{main_2}. 
Let us start with the following lemma. 

\begin{lemma}[\cite{Baik--Kim--Koberda}]
Let $A$ be a right-angled Artin group and $G$ a group. 
If $A$ is embedded in $G$, then $A$ is also embedded in each finite index subgroup of $G$.  
\end{lemma}

The following is the ``only if part" of Theorem \ref{main_1}. 

\begin{theorem}
If the braid group $B_{n}$ is virtually embedded in $\mathrm{Mod}(S_{g, p})$, then 
\begin{eqnarray*}
n \leq \left\{ \begin{array}{ll}
1 & ((g, p) \in \{ (0, 0), (0, 1), (0, 2), (0, 3) \}) \\
2 & ((g, p) \in \{ (0, 4), (1, 0), (1, 1) \}) \\
3 & ((g, p) \in \{ (0, 5), (1, 2) \}) \\
p-2 & (g=0, \ p \geq 6) \\
p + 1 & (g=1, \ p \geq 3) \\
2g & (g \geq 2, \ p=0)  \\
2g + p - 1 & (g \geq 2, \ p \geq 1). \\
\end{array} \right.
\end{eqnarray*}
 \label{braid_closed_mcg}
\end{theorem}
\begin{proof}
Suppose that $B_{n}$ is virtually embedded in $\mathrm{Mod}(S_{g, p})$. 
First, we treat the cases where surfaces have complexities less than $3$. 

{\bf Case of $(g, p) \in \{ (0, 0), (0, 1), (0, 2), (0, 3) \}$}. 
In this case $\mathrm{Mod}(S_{g, p})$ is a finite group, so there is no virtual embedding of $B_2$ into $\mathrm{Mod}(S_{g, p})$. 

{\bf Case of $(g, p) \in \{ (0, 4), (1, 0), (1, 1) \}$}. 
Since the mapping class group of $S_{g, p}^{b}$ is Gromov hyperbolic and since every finite index subgroup of $B_3$ contains a free abelian group of rank $2$, $B_3$ has no finite index subgroup embedded in the mapping class group. 

{\bf Case of $(g, p) \in \{ (0, 5), (1, 2) \}$}. 
Since $\mathbb{Z}^3$ is not embedded in $\mathrm{Mod}(S_{g, p})$, $B_4$ has no finite index subgroup embedded in $\mathrm{Mod}(S_{g, p})$. 
Hence, the desired inequality holds. 

In the remaining cases, we may assume that $n \geq 3$. 
Then, since $A(C_{n+1}^c) \times \mathbb{Z}$ is embedded in all finite index subgroups of $B_n$, $A(C_{n+1}^c) \times \mathbb{Z}$ is embedded in $\mathrm{Mod}(S_{g, p})$. 
We use Theorem \ref{cyc_plus_z_mcg} to obtain the desired inequalities. 

{\bf Case of $g=0$ and $p \geq 6$}. 
We have $n+1 \leq p-1$. 

{\bf Case of $g=1$ and $p \geq 3$}. 
We have $n+1 \leq p+2$. 

{\bf Case of $g \geq 2$ and $p=0$}.  
We have $n+1 \leq 2g+1$. 

{\bf Case of $g \geq 2$ and $p \geq 1$}. 
We have $n+1 \leq 2g+p$. 
\end{proof}

The following lemma is convenient. 

\begin{lemma}
Suppose $n \geq 4$. 
Suppose also that $B_n$ is virtually embedded in $\mathrm{Mod}(S_{g, p}^{b})$. 
Then, $A(C_{n+1}^c)$ is embedded in $\mathrm{Mod}(S_{g, p+b})$. 
 \label{cap_raag}
\end{lemma}
\begin{proof}
The kernel of the capping homomorphism $\mathrm{Mod}(S_{g, p}^{b}) \rightarrow \mathrm{Mod}(S_{g, p+b})$ is generated by the center. 
Since $C_{n+1}^c$ has no vertex adjacent to the other vertices, the center of $A(C_{n+1}^c)$ is trivial. 
This is an immediate corollary of Servatius' centralizer theorem \cite{Servatius} which implies the centralizer of a vertex $v$ is generated by $v$ and the vertices adjacent to $v$. 
Hence, we obtain an embedding of $A(C_{n+1}^c)$ into $\mathrm{Mod}(S_{g, p+b})$. 
\end{proof}

The following is the ``only if part" of Theorem \ref{main_2}. 

\begin{theorem}
Suppose that $b \geq 1$. 
Suppose also that the braid group $B_{n}$ is virtually embedded in $\mathrm{Mod}(S_{g, p}^{b})$. 
Then,   
\begin{eqnarray*}
n \leq \left\{ \begin{array}{ll}
2g + p + b & (g \geq 1, \ p+b \leq 2) \\
2g + p + b - 1 & (\mbox{otherwise}). \\
\end{array} \right.
\end{eqnarray*}
 \label{braid_boundary_mcg}
\end{theorem}
\begin{proof}
Suppose that $B_n$ is virtually embedded in $\mathrm{Mod}(S_{g, p}^{b})$. 

{\bf Case of $g=0, \ p \geq 0, \ b=1$}. 
In this case $\mathrm{Mod}(S_{g, p}^{1})$ is isomorphic to $B_p$. 
The inequality $n \leq p$ $(= 2g+p+b-1)$ immediately follows from the fact that $\mathbb{Z}^{l} \hookrightarrow B_m$ if and only if $l \leq m-1$. 

{\bf Case of $g=0$ and $p \geq 0, \ b \geq 2$}. 
Suppose that $p+b \leq 3$. 
Then, $(p, b) \in  \{ (0, 2), (1, 2), (0, 3) \}$, and so $\mathrm{Mod}(S_{0, p}^{b})$ is a free abelian group. 
Since $B_3$ has no finite index abelian subgroup, $B_3$ is not embedded in $\mathrm{Mod}(S_{0, p}^{b})$ even virtually. 
Therefore $n \leq 2$ $(= 2g+p+b-1)$. 
Next, consider the case where $p+b = 4$. 
Since $\mathrm{Mod}(S_{0, 4})$ is virtually free, $\mathrm{Mod}(S_{0, p}^{b})$ is a central extension of a virtually free group. 
This implies that any centerless right-angled Artin subgroup of $B_n$ (embedded in $\mathrm{Mod}(S_{0, p}^{b})$) must be free. 
Since $B_4$ contains $A(P_4^{c})$ which is centerless but not free, we have $n \leq 3$. 
Suppose that $p+b \geq 5$. 
Assuming $n \geq 4$, we shall prove that $n \leq p+b-1$. 
By Lemma \ref{cap_raag}, $A(C_{n+1}^c)$ is embedded in $\mathrm{Mod}(S_{0, p+b})$. 
Theorem \ref{cyc_mcg} implies that $n+1 \leq p+b$. 

{\bf Case of $g = 1, \ p = 0, \ b \leq 2$}. 
In this case $\mathbb{Z}^{b+2}$ cannot be embedded in $\mathrm{Mod}(S_{1}^{b})$. 
Hence, no finite index subgroup of $B_{b+3}$ is embedded in $\mathrm{Mod}(S_{1}^{b})$. 

{\bf Case of $g \geq 2, \ p = 0, \ b \leq 2$}. 
We may assume that $n \geq 4$. 
By Lemma \ref{cap_raag}, $A(C_{n+1}^c)$ is embedded in $\mathrm{Mod}(S_{g, b})$. 
Since $1 \leq b \leq 2$, we have $n+1 \leq 2g+b+1$ by Theorem \ref{cyc_mcg}. 
When $b \geq 3$, we have $n+1 \leq 2g+b$. 

{\bf Case of $g \geq 1, \ p = 1, \ b = 1$}. 
We may assume that $n \geq 4$. 
By Lemma \ref{cap_raag} and Theorem \ref{cyc_mcg}, we have $n+1 \leq 2g+p+b+1$. 

{\bf Case of $g \geq 1, \ p+b \geq 3$}. 
We may assume that $n \geq 4$. 
By Lemma \ref{cap_raag} and Theorem \ref{cyc_mcg}, we have $n+1 \leq 2g+p+b$. 
\end{proof}

Below we summarize topological methods for constructing virtual embeddings of braid groups into the mapping class groups of surfaces in Theorems \ref{topological_condition} and \ref{ex_case}. 

\begin{theorem}
Suppose that either $b \neq 1$ or $p \neq 1$ holds. 
Let $S$ denote the surface $S_{g, p}^{b}$. 
Then the inequality in Theorems \ref{main_1} and \ref{main_2} holds if and only if there is at least one of the following: 
\begin{enumerate}
 \item[(1)] an extension $D \rightarrow S$ which is either trivial or hyperbolic, 
 \item[(2)] an extension $\tilde{D} \rightarrow S$ which is trivial, annular, or hyperbolic,   
 \item[(3)] a composition of extensions $D \rightarrow P \rightarrow S$, where the middle surface $P$ is a sphere, the first extension is hyperbolic and the second is pseudo-annular. 
\end{enumerate}
Here, $D$ is the defining surface $S_{0, n}^{1}$ of $B_n$ and $\tilde{D}$ is a double branched cover of $D$, which is homeomorphic to either $S_{k}^{2}$ $(k=\frac{n}{2}-1)$ or $S_{l}^{1}$ $(l=\frac{n-1}{2})$ according to whether $n$ is even or odd. 
 \label{topological_condition}
\end{theorem}
\begin{proof}
Suppose that there is at least one of (1), (2) or (3). 
Then we can find a virtual embedding of the braid group $B_n$ described in Section \ref{const_embeddings}. 
Hence, Theorems \ref{braid_closed_mcg} and \ref{braid_boundary_mcg} imply that the inequality on $n$ in Theorems \ref{main_1} and \ref{main_2} holds. 

We shall prove the ``only if part" of this theorem. 

{\bf Case of $S \cong S_{0, p}$}. 
Suppose $n \leq p-2$. 
Then, there is a hyperbolic extension of $D$ to $S$, so we have (1). 

{\bf Case of $S \cong S_{1, p}$}. 
Suppose $n \leq p+1$. 
Consider a pseudo-annular extension $S_{0, p+1}^{1} \rightarrow S$. 
This is obtained from the compactification $S_{0}^{p+2}$ of $S_{0, p+1}^{1}$, a single annulus and $p$ copies of a once-punctured disks. 
Then, this pseudo-annular extension together with a hyperbolic extension $D \rightarrow S_{0, p+1}^{1}$ satisfies (3). 

{\bf Case of $S \cong S_{g}$ with $g \geq 2$}. 
Suppose $n \leq 2g$. 
In this case we have an annular extension $S_{g-1}^{2} \rightarrow S$. 
Since $n \leq 2g$, the natural numbers $k$ and $l$ defined in the assertion are both less than $g$. 
Hence, we have an extension $\tilde{D} \rightarrow S_{g-1}^{2}$ which is either trivial (this case occurs if $n = 2g$) or hyperbolic. 
So we have (2). 

{\bf Case of $S \cong S_{g, p}$ with $g\geq 2$ and with $p \geq 1$}. 
Suppose $n \leq 2g+p-1$. 
Consider a pseudo-annular extension $S_{0, 2g+p-1}^{1} \rightarrow S$ obtained from the compactification $S_{0}^{2g+p}$, $g$ copies of an annulus and $p$ copies of a once-punctured disk. 
Since $n \leq 2g+p-1$, there is a hyperbolic extension $D \rightarrow S_{0, 2g+p-1}^{1}$. 
So we have $(3)$. 

{\bf Case of $S \cong S_{0, p}^{1}$}. 
Suppose $n \leq p$. 
Then we can find an extension $D \rightarrow S$ that is either trivial or hyperbolic. 
So we have (1). 

{\bf Case of $S \cong S_{0, p}^{b}$ with $b \geq 2$}. 
Suppose $n \leq p+b-1$. 
Then, there is a pseudo-annular extension $S_{0, p+b-1}^{1} \rightarrow S$ obtained from $S_{0}^{p+b}$, $b-1$ copies of an annulus, and $p$ copies of a once-punctured disks. 
Since $n \leq p+b-1$, there is a hyperbolic (or trivial) extension $D \rightarrow S_{0, p+b-1}^{1}$. 
So we have (3). 

{\bf Case of $S \cong S_{g}^{b}$ with $g \geq 1$ and with $1 \leq b \leq 2$}. 
Suppose $n \leq 2g+b$. 
Then, $S$ is a double branched covering of $S_{0, 2g+b}^{1}$. 
From a hyperbolic (resp.\ trivial) extension $D \rightarrow S_{0, 2g+b}^{1}$, we have a hyperbolic (resp.\ trivial) extension $\tilde{D} \rightarrow S$. 
So we have (2). 

{\bf Case of $S \cong S_{g, p}^{b}$ with $g \geq 1$ and with $p+b \geq 3$}. 
Suppose $n \leq 2g+p+b-1$. 
Consider a pseudo-annular extension $S_{0, 2g+p+b-1}^{1} \rightarrow S$ obtained from $S_{0}^{2g+p+b}$, $g$ copies of an annulus and $p$ copies of a once-punctured disk. 
A combination of this pseudo-annular extension and a hyperbolic extension $D \rightarrow S_{0, 2g+p+b-1}^{1}$ satisfies (3) and completes the proof. 
\end{proof}

Note that Theorem \ref{topological_condition} proves the ``if part" of Theorems \ref{main_1} and \ref{main_2} when $S \not \cong S_{g, 1}^{1}$. 

\begin{theorem}
Suppose that $g \geq 1$. 
Then $n \leq 2g+2$ if and only if one of the following holds: 
\begin{enumerate}
 \item[(1)] $\tilde{D} \cong S_{g}^{2} \overset{\mathrm{cap}}{\rightarrow} S_{g, 1}^{1}$ or
 \item[(2)] there is a hyperbolic extension $\tilde{D} \rightarrow S_{g}^{2}$ together with $S_{g}^{2} \overset{\mathrm{cap}}{\rightarrow} S_{g, 1}^{1}$. 
\end{enumerate}
 Here, $\tilde{D}$ is a Birman--Hilden cover of $S_{0, n}^{1}$, which is homeomorphic to either $S_{k}^{2}$ $(k=\frac{n}{2}-1)$ or $S_{l}^{1}$ $(l=\frac{n-1}{2})$ according to whether $n$ is even or odd, and $S_{g}^{2} \overset{\mathrm{cap}}{\rightarrow} S_{g, 1}^{1}$ is an admissible extension obtained from $S_{g}^{2}$ and a once-punctured disk by gluing along a boundary component of $S_{g}^{2}$ and the boundary of the once-punctured disk. 
 \label{ex_case}
\end{theorem}
\begin{proof}
Suppose that (1) or (2) holds. 
Then the entire braid group $B_n$ is embedded in $\mathrm{Mod}(S_{g}^{2})$. 
Hence, Theorem \ref{braid_boundary_mcg} implies $n \leq 2g+2$. 

Next, let us assume that $n \leq 2g+2$ and $\tilde{D} \not\cong S_{g}^{2}$ (so (1) does not hold). 
Since $\tilde{D} \not\cong S_{g}^{2}$, we have $n \leq 2g+1$. 
So we can find a hyperbolic extension $\tilde{D} \rightarrow S_{g}^{2}$. 
\end{proof}

Recall that $B_n$ is virtually embedded in $\mathrm{Mod}(S_{g, 1}^{1})$ only if $n \leq 2g+2$ by Theorem \ref{braid_boundary_mcg}. 
Let us prove the converse in accordance with Theorem \ref{ex_case}. 

\begin{theorem}
$B_n$ is virtually embedded in $\mathrm{Mod}(S_{g, 1}^{1})$ if $n \leq 2g+2$. 
 \label{ex_case_2}
\end{theorem}
\begin{proof}
Suppose $n \leq 2g+2$. 
Then we have an embedding $B_{2g+2} \hookrightarrow \mathrm{Mod}(S_{g}^{2})$ induced by a double branched covering. 
Moreover, from Proposition \ref{for_ex_case} we have an embedding $B_{2g+2} \rightarrow \mathrm{Mod}(S_{g, 1}^{1})$. 
Now the embedding $B_n \hookrightarrow B_{2g+2}$ induced by a hyperbolic extension completes the proof. 
\end{proof}

Consequently, the ``if part" of Theorems \ref{main_1} and \ref{main_2} is also true. 

Furthermore, our argument allows us to obtain the following results on embeddings of pure braid groups.  

\begin{theorem}
$PB_{n} \hookrightarrow \mathrm{Mod}(S_{g, p})$ if and only if 
\begin{eqnarray*}
n \leq \left\{ \begin{array}{ll}
1 & ((g, p) \in \{ (0, 0), (0, 1), (0, 2), (0, 3) \}) \\
2 & ((g, p) \in \{ (0, 4), (1, 0), (1, 1) \}) \\
- \chi_{g, p} & (g=0, \ p \geq 5) \\
2 - \chi_{g, p} & (g \geq 2, \ p=0)  \\
1 - \chi_{g, p} & (\mbox{otherwise}). \\
\end{array} \right.
\end{eqnarray*}
 \label{pure_main_1}
\end{theorem}

\begin{theorem}
Suppose $b \geq 1$. 
Then $PB_{n} \hookrightarrow \mathrm{Mod}(S_{g, p}^{b})$ if and only if 
\begin{eqnarray*}
n \leq \left\{ \begin{array}{ll}
2 - \chi_{g, p}^{b} & (g \geq 1, \ p+b \leq 2) \\
1 - \chi_{g, p}^{b} & (\mbox{otherwise}). \\
\end{array} \right.
\end{eqnarray*}
 \label{pure_main_2}
\end{theorem}


\begin{thebibliography}{99}
\bibitem{Aramayona--Leininger--Souto06}
J.\ Aramayona, C.\ Leiniger and J.\ Souto, {\it Injections of mapping class groups}, Geom.\ Topol.\ {\bf 13} (2009), no.\ 5, 2523--2541. 

\bibitem{Aramayona--Souto}
J.\ Aramayona and J.\ Souto, {\it Homomorphisms between mapping class groups}, Geom.\ Topol.\ {\bf 16} (2012), no.\ 4, 2285--2341. 

\bibitem{Baik--Kim--Koberda}
H.\ Baik, S.\ Kim and T.\ Koberda, {\it Unsmoothable group actions on compact one-manifolds}, J.\ Eur.\ Math.\ Soc.\ {\bf 21} (2019), 2333-2353. 

\bibitem{Behrstock--Kleiner--Minsky--Mosher}
J.\ Behrstock, B.\ Kleiner, Y.\ Minsky and L.\ Mosher, {\it Geometry and rigidity of mapping class groups}, Geom.\ Topol.\ {\bf 16} (2012), no.\ 2, 781--888. 

\bibitem{Birman--Hilden}
J.\ Birman and H.\ Hilden, {\it On isotopies of homeomorphisms of Riemann surfaces}, Ann.\ of Math.\ (2) {\bf 97} (1973), no.\ 3, 424--439. 

\bibitem{Birman--Lubotzky--McCarthy}
J. Birman, A.\ Lubotzky and J.\ McCarthy, {\it Abelian and solvable subgroups of the mapping class groups}, Duke Math.\ J.\ {\bf 50} (1983), no.\ 4, 1107--1120. 

\bibitem{Castel16}
F.\ Castel, {\it Geometric representations of the braid groups}, Asterisque, No. 378 (2016), vi+175 pp. ISBN: 978-2-85629-835-0. 

\bibitem{Clay--Leininger--Margalit}
M.\ Clay, C.\ Leininger and D.\ Margalit, {\it Abstract commensurators of right-angled Artin groups and mapping class groups}, Math.\ Res.\ Lett.\ {\bf 21} (2014), no.\ 3, 461--467. 

\bibitem{Hamenstaedt}
U.\ Hamenst\"adt, {\it Geometry of the mapping class groups III: Quasi-isometric rigidity}, preprint, arXiv: math/051242. 

\bibitem{Harer}
J.\ Harer, {\it The virtual cohomological dimension of the mapping class group of an orientable surface}, Invent.\ Math.\ {\bf 84} (1986), 157--176. 

\bibitem{Ivanov--McCarthy}
N.\ Ivanov and J.\ McCarthy, {\it On injective homomorphisms between Teichm\"uller modular groups. I}, Invent.\ Math. {\bf 135} (1999), 425--486. 

\bibitem{Katayama--Kuno18}
T.\ Katayama and E.\ Kuno, {\it The RAAGs on the complement graphs of path graphs in mapping class groups}, preprint, available at arXiv: 1804.03470v2. 

\bibitem{Kim--Koberda}
S.\ Kim and T.\ Koberda, {\it An obstruction to embedding right-angled Artin groups in mapping class groups}, Int.\ Math.\ Res.\ Not.\ {\bf 2014} (2014), no.\ 14, 3912--3918. 

\bibitem{Koberda}
T.\ Koberda, {\it Right-angled Artin groups and a generalized isomorphism problem for finitely generated subgroups of mapping class groups}, Geom.\ Funct.\ Anal.\ {\bf 22} (2012), 1541--1590. 

\bibitem{Chen--Mukherjea}
L.\ Chen and A.\ Mukherjea, {\it From braid groups to mapping class groups}, preprint, available at arXiv: 2011.13020. 

\bibitem{Margalit--Winarski}
D.\ Margalit and R.\ Winarski, {\it The Birman--Hilden theory}, Celebratio Mathematica (2017) and Bulletin of the London Mathematical Society, {\bf 53} (2021), no.\ 3, 643--659. 

\bibitem{Paris--Rolfsen00}
L.\ Paris and D.\ Rolfsen, {\it Geometric subgroups of mapping class groups}, J.\ Reine Angew.\ Math.\ {\bf 521} (2000), 47--83.  

\bibitem{Servatius}
H.\ Servatius, {\it Automorphisms of graph groups}, J.\ Algebra, {\bf 126} (1989), 34--60

\bibitem{Shackleton}
K.\ Shackleton, {\it Combinatorial rigidity in curve complexes and mapping class groups}, Pac.\ J.\ Math.\ {\bf 230} (2007), no.\ 2, 217--232. 

\end{thebibliography}
\end{document}